\documentclass[a4paper,12pt]{amsart}

\usepackage{amssymb,mathrsfs, color,hyperref,nicefrac,dsfont,csquotes,enumitem,
xparse,aliascnt,geometry,tikz,pgfplots}
\usetikzlibrary{patterns}
\pgfmathdeclarefunction{gauss}{2}{%
  \pgfmathparse{1/(#2*sqrt(2*pi))*exp(-((x-#1)^2)/(2*#2^2))}%
}

\usepackage[showonlyrefs]{mathtools}

\setlist[itemize]{leftmargin=*}
\setlist[enumerate]{leftmargin=*,wide, labelwidth=!, labelindent=3em}
\setlist[enumerate,1]{label=(\roman*),ref=(\roman*),wide, labelwidth=!, labelindent=0pt}
\setlist[description]{leftmargin=*} 
 
 \numberwithin{equation}{section}

\renewcommand\nicefrac[2]{#1/#2}
\newcommand\tforall{\quad\text{ for any }}

\newcommand\1{\mathds{1}}
\newcommand\dd{\, d}
\newcommand\Unif{\operatorname{Unif}}

\newcommand\R{\mathbb{R}}
\newcommand\N{\mathbb{N}}
\renewcommand{\r}{\mathbb{R}}
\newcommand{\rd}{\mathbb{R}^d}

\newcommand{\E}{\mathbb{E}}

\newcommand\MALA{MALA }
\newcommand\wc{\cdot}
\renewcommand{\P}{\mathbb{P}}
\renewcommand{\tilde}{\widetilde}
\renewcommand{\epsilon}{\varepsilon}

\newcommand{\ob}{\overline{\beta}}
\newcommand{\ua}{\underline{\alpha}}
\newcommand{\up}{\underline{\pi}}

\newtheorem{prop}{Proposition}[section]
\newtheorem{thm}[prop]{Theorem}
\newtheorem{theorem}[prop]{Theorem}

\newtheorem{lemma}[prop]{Lemma}

\theoremstyle{definition}
\newtheorem{remark}[prop]{Remark}
\newtheorem{example}[prop]{Example}

\title[Quantitative contraction rates for Markov chains]{Quantitative contraction rates for Markov chains on general state spaces}
\author{Andreas Eberle}
\address{Universit\"at Bonn, Institut f\"ur Angewandte Mathematik,
  Endenicher Allee 60, 53115 Bonn, Germany}
\email{eberle@uni-bonn.de}
\urladdr{http://www.uni-bonn.de/$\sim$eberle}
\author{Mateusz B. Majka}
\address{Department of Mathematics, King's College London, Strand, London WC2R 2LS, UK, and, The Alan Turing Institute, 96 Euston Road, London NW1 2DB, UK}
\curraddr{Department of Statistics, University of Warwick, Coventry CV4 7AL, UK}
\email{mateusz.b.majka@gmail.com}
\urladdr{https://sites.google.com/site/mateuszbmajka/}

\subjclass[2010]{60J05, 60J22, 65C05, 65C30, 65C40}
\keywords{Markov chains, Wasserstein distances, quantitative bounds, couplings, Euler schemes, Metropolis algorithm.} 
\thanks{AE is supported by the German Science foundation through the Hausdorff Center for Mathematics. MM is supported by the EPSRC grant no.\ EP/P003818/1. Both authors acknowledge partial support by the grant 346300 for IMPAN from the Simons Foundation and the matching 2015-2019 Polish MNiSW fund.}

\begin{document}

\begin{abstract}
We investigate the problem of quantifying contraction coefficients of Markov transition kernels in Kantorovich ($L^1$ Wasserstein) distances. For diffusion processes, relatively precise quantitative bounds on contraction rates have recently been derived by combining appropriate couplings with carefully designed Kantorovich distances.
In this paper, we 
partially carry over this approach from diffusions to Markov chains.
We derive quantitative lower bounds on contraction rates for Markov chains on general state spaces that are 
powerful if the dynamics is dominated by small local moves. For Markov chains 
on $\mathbb R^d$ with isotropic transition kernels, the general bounds can be used efficiently together with a coupling that combines maximal and reflection coupling.
The results are
applied to Euler discretizations of stochastic differential equations with non-globally contractive drifts, and to the Metropolis adjusted Langevin algorithm for sampling
from a class of probability measures on high dimensional state spaces that are not globally log-concave.
\end{abstract}

\maketitle

\section{Introduction}

In recent years, convergence bounds for Markov processes in
Kantorovich ($L^1$ Wasserstein) distances have emerged as a powerful alternative to more traditional approaches based on the total variation distance \cite{MT}, spectral gaps and $L^2$ bounds
\cite{JSTV, DSC, DLP}, or entropy estimates \cite{JSTV,DSC,BGL}. In particular, Hairer, Mattingly and Scheutzow have developed an analogue to Harris' Theorem assuming only
local strict contractivity in a Kantorovich distance on the \enquote{small} set
and a Lyapunov condition combined with non-strict contractivity outside, cf.\ \cite{HM,HMS}. Meanwhile there have been numerous extensions and applications of their result \cite{HSV, CH, Butkovsky, DFM}. 

In \cite{JO}, Joulin and Ollivier have shown that
strict Kantorovich contractivity of the transition kernel implies bounds 
for the variance and concentration estimates for ergodic averages of
a Markov chain. 
Their results have since been extended to cover more general frameworks by Paulin \cite{Paulin}.
More recently, Pillai and Smith \cite{PS} as well as Rudolf and Schweizer \cite{RS} have developed a 
perturbation theory for Markov chains that are contractive in a
Kantorovich distance, cf.\ also Huggins and Zou \cite{HugginsZou} as well as Johndrow and Mattingly \cite{JM} for related results. These works show that variants of the results in \cite{JO} carry over to perturbations
of the original chain, thus paving the way for a much broader range
of applications.

All the works mentioned above assume that, at least locally, strict contractivity 
holds w.r.t.\ an $L^1$ Wasserstein distance based on some 
underlying distance function on the state space of the Markov chain.
The contraction rate is the key quantity in the resulting bounds, and
it is hence important to develop applicable methods for quantifying
contraction rates. 

Contractivity with respect to the $L^1$ Wasserstein distance based
on the Euclidean distance in $\rd$ is sometimes interpreted as non-negative
Ricci curvature of the Markov chain w.r.t.\ this metric \cite{RenesseSturm,JO,Ollivier}.
This is a strong condition that is often not satisfied in applications.
However, in many cases it is still possible to obtain contractivity with respect to a Kantorovich distance in which the underlying distance function has been modified accordingly. This allows for applying the results from \cite{JO} to a significantly broader class of examples. For diffusion processes, a corresponding approach to
quantitative contraction rates in appropriately designed metrics has been developed systematically in recent years in a series of papers \cite{EberleCR,Eberle2015,Zimmer,EGZ,EZ}, see also \cite{CWa,CW,WangNeumann} for previous results. The approach has been extended 
to L\'evy driven SDEs in \cite{Majka17, Majka}, see also \cite{LuoWang,JianWangBernoulli}.

Below we propose a corresponding approach for Markov chains on general metric state spaces. The approach is powerful in situations where the dynamics is dominated by small, local moves. This will be demonstrated below for Euler schemes for non-globally contractive stochastic differential equations, as well as for the Metropolis-adjusted Langevin Algorithm (MALA). In these cases, the Ricci curvature condition required in \cite{JO} is not satisfied in the standard $L^1$ Wasserstein distance and hence the construction of an alternative metric is required. For dynamics dominated by large or global moves, our approach does not apply in the form presented here. Sometimes, related approaches
can be used nevertheless, see e.g.\ \cite{BEZ} for the construction of a contractive distance for  Hamiltonian Monte Carlo.

\section{Main results}

Let $p(x, dy)$ be a Markov transition kernel on a separable metric 
space $(S,d)$. To study contraction properties of $p$ we 
construct distance functions $\rho :S\times S\to [0,\infty )$ by transforming the metric $d$ in an appropriate way. Note that if $f_0: [0,\infty) \to [0,\infty)$ is a concave, increasing function with $f_0(0) = 0$ and $f_0(r) > 0$ for $r \neq 0$, then $\rho(x,y) = f_0(d(x,y))$ is a metric on $S$. More generally,
let $a,\delta \ge 0$ be non-negative constants, and let $V\colon S\to [0,\infty )$ be a non-negative measurable function on $S$. We set
\[
f:=f_0+a\1_{(0,\infty )} ,
\]
and we consider distance functions of the form
\begin{equation}
\label{eq:1}
\begin{split}
\rho (x,y) &= f(d(x,y))+\delta\, (V(x)+V(y))\1_{x\neq y}\\
 &=  f_0(d(x,y))+(a+\delta V(x)+\delta V(y))\1_{x\neq y} \,.
\end{split}
\end{equation}
We assume that $f_0$ is continuous and in some of our results we will choose $a=0$ and $V\equiv 0$. {Similar but slightly different classes of distance functions have been used e.g.\ in \cite{HM2,HMS,Butkovsky,EGZ} to study properties of Markov chains and diffusion processes.}

For probability measures
$\mu$ and $\nu$ on $S$, the Kantorovich distance ($L^1$
Wasserstein distance) $\mathcal W_\rho (\mu ,\nu )$ based on the
underlying distance function $\rho$ is defined as
\begin{equation}
\label{eq:2}
\mathcal W_\rho (\mu ,\nu )\ =\inf_{X\sim\mu ,Y\sim\nu} \E[\rho (X,Y)] .
\end{equation}
Here the infimum is over all couplings of $\mu$ and $\nu$, i.e.,
over all random variables $X,Y$ defined on a common probability
space $(\Omega ,\mathcal A,\P )$ such that $\mathbb{P}\circ X^{-1}=\mu$
and $\mathbb{P}\circ Y^{-1}=\nu $. 

For $f_0\equiv 0$, $a=1$ and $V\equiv 0$, $\mathcal W_\rho$ coincides with the total variation distance $d_{\text{TV}}(\mu,\nu)$ (or with $d_{\text{TV}}(\mu,\nu)/2$, depending on the convention used in the definition of the total variation distance),
 whereas for $f_0(r)=r$, $a=0$ and $V\equiv 0$, $\mathcal W_\rho$ is the
standard $L^1$ Wasserstein distance $\mathcal W_d$ on $(S,d)$.
The distance functions we consider are in between these two extremes. Notice, however, that if $a>0$ then
\begin{equation}
\label{eq:3}
d_{\rm TV}(\mu ,\nu )\le\ a^{-1}  \mathcal W_\rho (\mu ,\nu ),
\end{equation}
and if $f(r)\ge br$ for some constant $b>0$ then
\begin{equation}
\label{eq:4}
\mathcal W_d (\mu ,\nu )\le\ b^{-1}  \mathcal W_\rho (\mu ,\nu ),
\end{equation}
Therefore, in these cases, contraction properties w.r.t.\ $\mathcal W_\rho$ directly
imply upper bounds for the total variation and $L^1$ Wasserstein
distances w.r.t.\ the metric $d$.\medskip

We now assume that we are given a {\em Markovian coupling} of the
transition probabilities $p(x,\wc )$ ($x\in S$) in the form of 
measurable maps $X',Y':\Omega\to S$, defined on a measurable space $(\Omega ,\mathcal A)$, and a probability kernel $(x,y,A)\mapsto \P_{x,y}(A)$ from $S\times
S\times\Omega $ to $[0,1]$ such that for any $x,y\in S$,
\begin{equation}
\label{eq:4a}
X'  \sim  p(x,\wc )\qquad \mbox{and}\qquad
Y' \sim  p(y,\wc )\qquad \mbox{under }\P_{x,y}.
\end{equation}
For probability measures $\mu$ on $S$ and $\gamma$ on $S\times S$ let
$(\mu p)(B) = \int \mu(d x) p(x,B)$ for $B\in\mathcal B(S)$, and $\P_\gamma (A )  =  \int\gamma(d x\, d y) 
\P_{x,y}(A )$ for $A\in\mathcal A$. Note that if $\gamma$ is a 
coupling of two probability measures $\mu$ and $\nu$ on $S$, then
under $\P_\gamma$ the joint law of $(X',Y')$ is a coupling of the probability measures $\mu p$ and $\nu p$, i.e.,
\begin{equation}\label{eq:4b}
X'\sim \mu p\qquad\mbox{and}\qquad Y'\sim\nu p\qquad\mbox{under }\P_\gamma .
\end{equation}
Our goal is to derive explicit bounds of the form
\begin{equation}
\label{eq6}
\E_{x,y}[\rho (X',Y')] \le (1-c)  \rho (x,y)\qquad\mbox{for any }
x,y\in S,
\end{equation}
where $c$ is a strictly positive constant. Here the choice of the metric $\rho$ is adapted in order to maximize the value of $c$ in
our bounds. If \eqref{eq6} holds, then the transition kernel $p$ is
a strict contraction w.r.t.\ the distance $\mathcal W_\rho$.

\begin{lemma}\label{lem:a}
Suppose that \eqref{eq6} holds for all $x,y\in S$. Then
\begin{equation}
\label{eq7}
\mathcal W_\rho (\mu p,\nu p)\le (1-c)  \mathcal W_\rho (\mu ,\nu )\qquad\mbox{for all }\mu ,\nu\in {\mathcal P}(S).
\end{equation}
\end{lemma}

\begin{proof}
Let $\mu$ and $\nu$ be probability measures on $S$ and suppose that $\gamma $ is a coupling of $\mu$ and $\nu$. Then, under
$\P_\gamma $, the joint law of $(X',Y')$ is a coupling of
$\mu p$ and $\nu p$. Therefore by \eqref{eq6}, 
\begin{eqnarray*}
\mathcal W_\rho (\mu p ,\nu p) &\le &\E_\gamma [\rho (X',Y')]
=\int \E_{x,y}[\rho (X',Y')]  \gamma ( d x  \dd y)\\
&\le & (1-c) \int\rho (x,y) \gamma ( d x  \dd y).
\end{eqnarray*}
The assertion follows by taking the infimum over all couplings of
$\mu $ and $\nu$.
\end{proof}

In the terminology of Joulin and Ollivier \cite{JO}, \eqref{eq7} says
that the Markov chain has a Ricci curvature lower bound $c$ on the
metric space $(S,\rho )$. By general results, such a bound has many important consequences including quantitative convergence to a 
unique equilibrium \cite{EberleMP}, upper bounds on biases and variances as well as concentration inequalities for ergodic averages
\cite{JO,Paulin}, a central limit theorem for ergodic averages \cite{KW}, 
robustness under perturbations \cite{PS,RS,HugginsZou,JM}, etc.
{However, in applications, it is usually not clear how to choose a distance function $\rho$ such that we have good bounds for $c$.}
This is the problem addressed in this paper 
for the case of a ``local dynamics'' where the Markov chain is mainly making 
``small'' moves. Depending on whether or not the probability measures $p(x,\cdot )$ and $p(y,\cdot )$ have a significant overlap for $x$ close to $y$, we suggest
two different approaches.

\subsection{Contractivity with positive coupling probability}
\label{sec:firstresult}

Our first two general results apply in situations where the probability measures $p(x,\wc )$ and $p(y,\wc )$ have  a significant overlap if
$x$ and $y$ are sufficiently close. In this case we can always consider a coupling $((X',Y'),\P_{x,y})$ of the transition probabilities such that
$\P_{x,y}[X'=Y']>0$ for $x$ close to $y$. This enables us to obtain strict
contractivity in metrics that have a total variation part, i.e., the function
$f$ defining the underlying distance has a discontinuity $a>0$ at $0$.

To state the results, we fix a positive constant $\epsilon >0$ and couplings
$((X',Y'),\P_{x,y})$ as above. For $x,y\in S$ we set
\begin{align}\label{eq:5}
r&=d(x,y),& R'&=d(X',Y'),&\Delta R&=R'-r,
\end{align}
and we define
\begin{align}
\label{eq:8}\beta (x,y) &= \E_{x,y}[\Delta R] ,\\
\label{eq:9}\alpha (x,y) &= \E_{x,y}\left[ \lvert(\Delta R)^-\wedge \epsilon \rvert^2\right] ,\\
\label{eq:10}\pi (x,y) &= \P_{x,y}[R'=0] ,
\end{align}
{where $(\Delta R)^-=\max (-\Delta R,0)$. In particular,}
\begin{equation}\label{eq:alphalb}
\alpha (x,y)\ \ge \ \mathbb{E}_{x,y}[ (\Delta R)^2 \1_{\{ R' \in (r - \varepsilon, r ) \} } ].
\end{equation}
{One can think of $\beta (x,y)$ as a drift for the coupling distance, whereas $\alpha (x,y)$ provides a lower bound for fluctuations that decrease the distance and $\pi (x,y)$ is the probability of coupling successfully in the next step.}
Suppose that there exist functions $\overline{\beta }:(0,\infty )\to\r $ and $\underline{\alpha},\underline{\pi}:(0,\infty )\to [0,\infty )$ such that
for any $r>0$ and $x,y\in S$ with $d(x,y)=r$,
\begin{equation}\label{eq:11}
\beta (x,y)\ \le\ \ob (r),\quad \alpha (x,y)\ \ge \ \ua (r),\quad \mbox{and}\quad
\pi (x,y)\ge \up (r).
\end{equation}
Hence $\ob (r)$ is an upper bound for the expectation of the increase 
$\Delta R$ of the distance during a single transition step of coupled Markov
chains with initial states $x$ and $y$ such that $d(x,y)=r$. Similarly,
$\ua (r)$ is a lower bound for distance decreasing fluctuations of $\Delta R$,
and $\up (r)$ is a lower bound for the coupling probability. We make the following assumptions on $\ua ,\ob $ and $\up$:

\begin{enumerate}[label=(A\arabic*)]
\item\label{enum:a1} There exists a positive constant $r_0\in (0,\infty )$ such that\smallskip
\begin{enumerate}
\item\label{enum:a1i} $\inf_{r\in (0,r_0]}\up (r)\ >\ 0$,\qquad {and}\smallskip
\item\label{enum:a1ii} $\inf_{r\in (r_0,s)}\ua (r)\ >\ 0$\quad for any $s\in (r_0,\infty )$.
\end{enumerate}\smallskip
\item\label{enum:a2} $\sup_{r\in (0,s)}\ob (r)<\infty$\quad for any $s\in (0,\infty )$.\smallskip
\item\label{enum:a3} $\limsup_{r\to\infty}r^{-1}\ob (r)\ <\ 0$.
\end{enumerate}

\begin{thm}\label{thm:1}
Suppose that \ref{enum:a1}, \ref{enum:a2} and \ref{enum:a3} are satisfied, and let 
\begin{equation}\label{eq:13}
\rho (x,y)=f(d(x,y)) \,,
\end{equation} 
where $f:[0,\infty )\to [0,\infty )$ is the concave increasing function defined in
\eqref{eq:5000} below. Then for any $x,y\in S$,
\begin{equation}
\label{eq:14}
\E_{x,y}[\rho (X',Y')]\le(1-c)  \rho (x,y) \,,
\end{equation}
where $c$ is an explicit strictly positive constant defined in \eqref{eq:5f} below.
\end{thm}

The proof is given in Section \ref{sec:proofs12}. Explicit
expressions for the function $f$ and the contraction rate $c$ depending only on $\ua ,\ob ,\up $ and $\epsilon$ are given in Subsection \ref{subsec:choiceOfMetricThm1}.  Although these 
expressions are somehow involved, they can be applied to derive 
quantitative bounds in concrete models. In particular, the
asymptotic dependence of the contraction rate on parameters of
the model can often be made explicit. This will be demonstrated for the Euler scheme in Section \ref{sec:euler}.

By Lemma \ref{lem:a}, Theorem \ref{thm:1} implies that the transition kernel $p$ 
is contractive with rate $c$ w.r.t.\ the $\mathcal W_\rho$ distance on probability
measures on $S$. Since the function $f$ defined in \eqref{eq:5000} is bounded from below by a multiple of both $\1_{(0,\infty )}$ and of the identity, the theorem yields
quantitative bounds for convergence to equilibrium both w.r.t.\ the total
variation and the standard $L^1$ Wasserstein distance.
\bigskip

The assumption \ref{enum:a3} imposed in Theorem \ref{thm:1} is sometimes too restrictive. By a modification of
the metric, it can be replaced by the following Lyapunov condition:

\begin{enumerate}[resume, label=(A\arabic*)]
\item\label{enum:a4}  There exist a measurable function $V:S\to [0,\infty )$ and
$C$, $\lambda\in (0,\infty )$ s.t.\smallskip
\begin{enumerate}
\item \label{enum:a4i}$pV\ \le\ (1-\lambda )  V  +  C,\qquad\mbox{and}$\smallskip
\item \label{enum:a4ii}$\inf\nolimits_{d(x,y)=r}\frac{V(x)+V(y)}{\ob (r)^+}\ \longrightarrow\ 
\infty\qquad\mbox{as }r\to\infty $.
\end{enumerate} 
\end{enumerate}
In \eqref{enum:a4ii} we use the convention that the value of the fraction is $+\infty$ if $\ob (r)\le 0$.

\begin{thm}\label{thm:2}
Suppose that \ref{enum:a1}, \ref{enum:a2} and \ref{enum:a4} are satisfied, and let 
\begin{equation}\label{eq:15}
\rho (x,y)=f(d(x,y))  +  \frac{M}{2C}  (V(x)+V(y)) \1_{x\neq y} \,,
\end{equation} 
where $f\colon [0,\infty )\to [0,\infty )$ is the concave increasing function defined in
\eqref{eq:6000} below, and the constant $M\in\r_+$ is defined in 
\eqref{eq:6b}. Then for any $x,y\in S$,
\begin{equation}
\label{eq:16}
\E_{x,y}[\rho (X',Y')]\le(1-c)  \rho (x,y) \,,
\end{equation}
where $c$ is an explicit strictly positive constant defined in \eqref{eq:6f} below.
\end{thm}

The proof of the theorem is given in Section \ref{sec:proofs12} and explicit expressions for the function $f$ and the constants $M$ and $c$
in terms of $\ua ,\ob ,\up ,\epsilon ,V,C$ and
$\lambda$ are provided in Subsection \ref{subsec:choiceOfMetricThm2}.

The idea of adding a Lyapunov function to the metric appears for example
in \cite{HairerConvergenceMP} and has been further worked out in the diffusion case in \cite{EGZ}. Theorem \ref{thm:2} can be seen as a more quantitative
version of {Theorem 4.8 in \cite{HMS}, which is an extension of the classical Harris' Theorem. 
Note, however, that contractivity in our result is expressed in an additive metric $\rho$, as opposed to the multiplicative semimetric used in \cite{HMS}; see also \cite{EGZ} for a more detailed discussion on these two types of metrics. 
An application of Theorem \ref{thm:2} to the Euler scheme is given in Theorem \ref{thm:8b} below.

\subsection{Contractivity without positive coupling probability}
\label{sec:secondresult}

The assumption that there is a significant overlap between the measures $p(x,\wc )$ and $p(y,\wc )$ for $x$ close to $y$ is sometimes too restrictive. For example, it may cause a bad dimension dependence of the resulting bounds in high dimensional applications. Therefore,
we now state an alternative contraction result that applies even
when $\pi (x,y)=0$ for all $x$ and $y$.

For any $r\in (0,\infty )$ we consider an interval near $r$ given by
\begin{equation}
\label{eq:17}
I_r=(r-l(r),r+u(r))
\end{equation}
where $l(r),u(r)\ge 0$ and $l(r)\le r$. Similarly as in \eqref{eq:8} and \eqref{eq:9}, we define
\begin{align}
\label{eq:18a}
\beta (x,y) &= \E_{x,y}[\Delta R],\\
\label{eq:19a}
\alpha (x,y) &= \E_{x,y}\left[\lvert(\Delta R\wedge u(r))\vee (-l(r))\rvert^2\right] ,
\end{align}
where $r,R'$ and $\Delta R$ are defined by \eqref{eq:5}.
{In particular,}
\begin{equation}\label{eq:alphalb2}
\alpha (x,y)\ \ge \ \mathbb{E}_{x,y}[ (\Delta R)^2 \1_{\{ R' \in I_r  \} } ].
\end{equation}
In Subsection \ref{sec:firstresult}, we have chosen $l(r)=\epsilon$ and
$u(r)=0$, i.e., $I_r=(r-\epsilon ,r)$. Now, we will assume instead that 
there is a finite constant $r_0>0$ such that
\begin{equation}
\label{eq:18}
u(r)\ =\ 0\quad \mbox{for }r\ge r_0,\qquad \mbox{and}\qquad
u(r)=r_0\quad\mbox{for }r<r_0.
\end{equation}
As above, we assume that there exist functions
$\ob :(0,\infty )\to\r $ and $\ua :(0,\infty )\to (0,\infty )$ such that
for any $r>0$ and $x,y\in S$ with $d(x,y)=r$,
\begin{equation}\label{eq:19}
\beta (x,y)\ \le\ \ob (r)\quad\mbox{and}\quad \alpha (x,y)\ \ge \ \ua (r).
\end{equation}
We now impose the following conditions on $\ua$ and $\ob$:

\begin{enumerate}[label=(B\arabic*)]
\item\label{enum:b1} $\inf\limits_{r\in (0,s)}\frac{\ua (r)}{r}\ >\ 0$\quad for any $s\in (0,\infty )$,
\item\label{enum:b2} $\sup\limits_{r\in (0,s)}\frac{\ob (r)}{\ua (r)}\ <\ \infty$\quad for any $s\in (0,\infty )$,
\item\label{enum:b3} $\limsup_{r\to\infty}r^{-1}\ob (r)\ <\ 0$.
\end{enumerate}

Thus we no longer assume a positive coupling probability for $r<r_0$.
Instead, we require in \ref{enum:b1} and \ref{enum:b2} that $\ua (r)=\Omega (r)$ and
$\nicefrac{\ob (r)}{\ua} (r)=O(1)$ as $r\downarrow 0$. These assumptions can be verified for
example for Euler schemes if the coupling is constructed carefully. We will do this
in Section
\ref{sec:euler} for Euler discretizations of SDEs with contractive drifts, whereas for more general drifts we
will follow a slightly different approach. 

\begin{thm}\label{thm:3}
Suppose that \ref{enum:b1}, \ref{enum:b2} and \ref{enum:b3} are satisfied, and let 
\begin{equation}\label{eq:20}
\rho (x,y)=f(d(x,y)) \,,
\end{equation} 
where $f\colon [0,\infty )\to [0,\infty )$ is a continuous concave increasing function satisfying $f(0)=0$ which is defined explicitly in
\eqref{eq:8b} below. Then for any $x,y\in S$,
\begin{equation}
\label{eq:21}
\E_{x,y}[\rho (X',Y')]\le(1-c)  \rho (x,y) \,,
\end{equation}
where $c$ is an explicit strictly positive constant defined in \eqref{eq:8f} below.
\end{thm}

The proof is given in Section \ref{sec:proof3}. Notice that in contrast to Theorem \ref{thm:1} and Theorem \ref{thm:2}, the function $f$ 
in Theorem \ref{thm:3} does not have a jump at $0$, i.e., the 
Kantorovich metric $\mathcal W_\rho$ does not contain a total 
variation part. This corresponds to the fact that under Assumptions
\ref{enum:b1}, \ref{enum:b2} and \ref{enum:b3}, it can not be expected in general that the coupled 
Markov chains meet in finite time.

\subsection{Stability under perturbations}\label{sec:perturbations}

Contractions in Kantorovich distances can sometimes be carried over 
to small perturbations of a given Markov chain. For instance,
in Subsection \ref{sec:\MALA} we will deduce contractivity for the 
Metropolis adjusted Langevin algorithm from corresponding
properties of the Euler proposal chain. Suppose as above that
$((X',Y'),\P_{x,y})$ is a Markovian coupling of the transition 
probabilities $p(x,\wc )$ and $p(y, \wc )$. Moreover, let
$((\tilde X,\tilde Y),{\P}_{x,y})$ be a corresponding coupling
of $\tilde p(x,\wc )$ and $\tilde p(y,\wc )$ for another (perturbed)
Markov transition kernel $\tilde p$ on $S$. Here we assume that for
given $x,y\in S$, $(X',Y')$ and $(\tilde X,\tilde Y)$ are defined on a
common probability space. We start with a simple observation.
If there exists a metric $\rho$ on $S$ and a constant $c\in (0,\infty )$
such that for $x,y\in S$,
\begin{eqnarray}
\label{eq:22}
\E_{x,y}[\rho (X',Y')] &\le & (1-c)  \rho (x,y),\qquad\qquad\mbox{and}\\
\label{eq:23}
{\E}_{x,y}[\rho (\tilde X,\tilde Y)] &\le &\E_{x,y}[\rho (X',Y')]  + 
\frac c2   \rho (x,y),\qquad\text{then}\\
\label{eq:24}
{\E}_{x,y}[\rho (\tilde X,\tilde Y)] &\le & \left(1-c/2\right)  \rho (x,y).
\end{eqnarray}
In applications it is often difficult or even impossible to verify Condition \eqref{eq:23} for $x$ very close to $y$. If $\P_{x,y}
[\tilde X=\tilde Y]>0$ for $x$ close to $y$, then this condition can be
relaxed.

\begin{thm}\label{thm:4}
Suppose that $\rho (x,y)=f(d(x,y))$ for a concave increasing contraction $f\colon [0,\infty )\to [0,\infty )$ satisfying $f(0)=0$. Suppose
that there exist constants $c,b,p,r_0\in [0,\infty )$ such that for
all $x,y\in S$,
\begin{align}
\label{eq:27}
\E_{x,y}[\rho (X',Y')] &\le  (1-c)  \rho (x,y),\\
\label{eq:28}
{\E}_{x,y}\left[ ( d(\tilde X,\tilde Y)-d(X',Y'))^+\right] &\le b  +  \frac c2  \rho (x,y)  ,\qquad\mbox{and}\\
\label{eq:29}
\P_{x,y}[\tilde X=\tilde Y] &\ge p\qquad\mbox{if }d(x,y)<r_0.
\end{align}
Assume that $p>0$, $b\le \nicefrac{cf(r_0)}{4}$, and let $\tilde\rho$ be the
metric defined by 
\begin{equation}
\label{eq:26}
\tilde\rho (x,y)=\rho (x,y)  + \frac{2b}{p}  \1_{x\neq y} . 
\end{equation}
Then
\begin{equation}
\label{eq:30}
{\E}_{x,y}\left[\tilde\rho (\tilde X,\tilde Y)\right] \ \le\  \left(1-\frac 18\min (c,2p)\right)  \tilde\rho (x,y)\qquad\mbox{for all }x,y\in S.
\end{equation}
\end{thm}

The proof is given in Section \ref{sec:proof4}. In Section \ref{sec:proofof\MALA} we will apply Theorem \ref{thm:4} to our results for the Euler scheme in order to obtain contractivity for the Metropolis adjusted Langevin algorithm (MALA).

{Note that Theorem \ref{thm:4} is related to the perturbation results in  \cite{PS,RS,JM}. In all these papers, a Kantorovich contraction in some metric is assumed for the initially given unperturbed Markov chain. Then, in \cite{PS}, the authors obtain bounds on the distance to equilibrium of a perturbed Markov chain in the same Kantorovich metric. In \cite{RS, JM}, the metric also remains unchanged, but the object of interest is a bound on the distance between a perturbed and the unperturbed chain. A related result in continuous time, giving bounds on the distance between invariant measures of a perturbed and an unperturbed diffusion, has been obtained in \cite{HugginsZou}. In contrast to these results, we consider a perturbed metric in Theorem \ref{thm:4}, but we obtain a stronger result showing that the perturbed Markov chain is again contractive w.r.t.\ the modified metric.}

\subsection{Application to Euler schemes}\label{sec:euler}

We now show how to apply the general methods developed above to Euler 
discretizations of stochastic differential equations of the form
\begin{equation}
\label{eqSDE}
\dd X_t=b(X_t)  \dd t  +  \dd B_t \,,
\end{equation} 
where $(B_t)_{t\ge 0}$ is a Brownian motion in $\r^d$, and $b:\r^d\to\r^d$ is
a Lipschitz continuous vector field. 
Quantifying contraction rates for Euler discretizations is important in connection with
the derivation of error bounds for the unadjusted Langevin algorithm (ULA), cf.\ \cite{Dalalyan,DurmusMoulinesULA0,DurmusMoulinesULA1,DurmusMoulinesULA2,DalalyanKaragulyan}
for corresponding results. Such applications of the techniques presented below will be discussed in detail in the upcoming paper \cite{MajkaMijatovicSzpruch} by the second author.
The transitions of the Markov chain for the Euler scheme with step size $h>0$ are given by
\begin{equation}
\label{eqES1}
x\mapsto\hat x  + \sqrt h  Z,\quad \text{where}\quad \hat x\coloneqq x  +  hb(x)\quad\text{and}\quad
Z\sim N(0,I_d).
\end{equation}
The corresponding transition probabilities are given by
\begin{equation}
\label{eqES1a}
p(x,\wc)\ =\ N(\hat x,hI_d)\qquad\mbox{for any }x\in\r^d,
\end{equation}
i.e., the transition density from $x$ is
\begin{equation}
\label{eqES1b}
p(x,x')\ =\ \phi_{\hat x,hI_d}(x')\ =\ (2\pi h)^{-\nicefrac{d}{2}}
\exp \left( -\frac{1}{2h}|x'-\hat x|^2\right) .
\end{equation}
In the case of $b\equiv 0$, the Markov chain is a Gaussian random walk
with transitions $x\mapsto x+\sqrt hZ$.

\subsubsection*{The coupling} 
For $x,y\in\r^d$ let
\begin{equation}
\label{equnit}
\hat e\ =\ \frac{\hat x-\hat y}{|\hat x-\hat y|}\qquad\mbox{if }\hat x\neq
\hat y,\qquad \hat e=
 0 \quad\text{otherwise. }
\end{equation}
We consider the coupling of two transitions of the
Euler chain from $x$ and $y$ respectively given by
\begin{equation}
\label{eq:0c}
X' \ =\ \hat x  + \sqrt hZ,\qquad  Y'\ =\
\begin{cases}
X' &\mbox{if }U\le\nicefrac{\phi_{\hat y,hI}(X')}{\phi_{\hat x,hI}(X')},\\
Y_{\rm refl}^\prime &\mbox{otherwise,}
\end{cases}
\end{equation}
where $Z\sim N(0,I_d)$ and $U\sim \Unif  (0,1)$ are independent
random variables, and
\begin{equation}
\label{eqRC}
Y_{\rm refl}^\prime\ =\ \hat y  +  \sqrt h  (I_d-2\hat e\hat e^T)Z
\end{equation}
is obtained by adding to $\hat y$ the increment $\sqrt hZ$ added
to $\hat x$, reflected at the hyperplane between $\hat x$ and
$\hat y$. \smallskip

\begin{figure}
	\centering
	\begin{tikzpicture}[->,shorten >=1pt,thick]
	\begin{axis}[domain=0:10, samples=100,
	axis lines*=left, xlabel={}, ylabel={},
	every axis y label/.style={at=(current axis.above origin),anchor=south},
	every axis x label/.style={at=(current axis.right of origin),anchor=west},
	height=5cm, width=12cm,
	xtick={4,6.5}, ytick=\empty,
	xticklabels={{$\hat{x}$},{$\hat{y}$}},
	enlargelimits=false, clip=false, axis on top,
	grid = major
	]
	\draw[] (400,0) circle [radius=2pt]; 
	\draw[] (650,0) circle [radius=2pt]; 
	\draw[dashed] (480,0) -- (480,286);
	\draw[fill=black] (480,0) circle [radius=2pt]; 
	\node[below] at (480,0) {$X'$};
	\draw[fill=black] (570,0) circle [radius=2pt]; 
	\node[below] at (570,0) {$Y_{\rm refl}^\prime$};
	\draw[fill=black] (480,93) circle [radius=2pt]; 
	\node[left] at (0,93) {$\phi_{\hat{y},hI}(X')$};
	\draw[fill=black] (480,286) circle [radius=2pt]; 
	\node[left] at (0,286) {$\phi_{\hat{x},hI}(X')$};
	\addplot [very thick,black, pattern=north west lines] {gauss(4,1)};
	\addplot [very thick,black, pattern=north east lines] {gauss(6.5,1)};
	\draw[dashed] (0,286) -- (480,286);
	\draw[dashed] (0,93) -- (480,93);
	\coordinate (a) at (3.8,.1);
	\coordinate (b) at (6.7,.1);
	\coordinate (c) at (5.25,.07);
	\end{axis}
\end{tikzpicture}
\caption{Construction of the coupling of $p(x,\cdot )$ and $p(y,\cdot )$: Given 
	the value of $X'$, we set $Y'=X'$ with the maximal probability $\min (1,\nicefrac{p(y,X')}{p(x,X')})$, and $Y'=Y_{\rm refl}^\prime$ otherwise.}
\label{fig:coupling}
\end{figure}
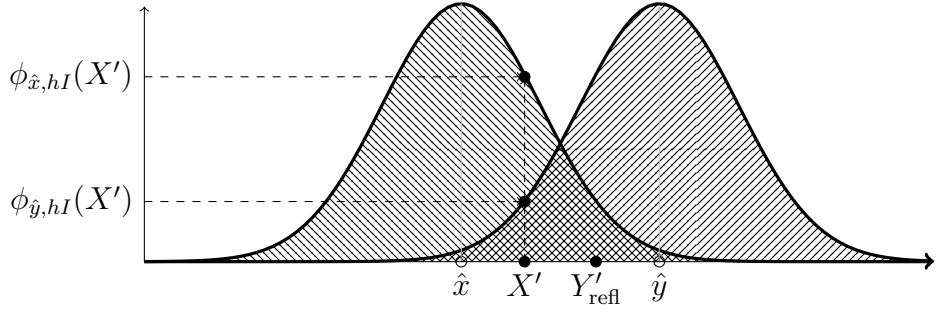
Both $(X',Y_{\rm refl}^\prime )$ and $(X',Y')$ are couplings of the probability measures $p(x,\wc)$ and $p(y,\wc )$.
For the coupling \eqref{eq:0c}, $Y'=X'$ with the maximal probability
$\min (1,\nicefrac{p(y,X')}{p(x,X')})$. Furthermore, in the case where $Y'\neq
X'$, the coupling coincides with the reflection coupling, i.e.,
$Y'=Y_{\rm refl}^\prime$. The resulting combination of reflection
coupling and maximal coupling is an optimal coupling of the Gaussian measures
$p(x,\wc )$ and $p(y,\wc )$ w.r.t.\ any Kantorovich distance
based on a metric $\rho (x,y)=f(\lvert x-y\rvert)$ with $f$ concave, cf.\
\cite{McCann} for the one-dimensional case.
We will not use the optimality here, but it shows that \eqref{eq:0c} is an appropriate coupling to consider if we
are interested in contraction properties for single transition steps of the 
Markov chain.

\begin{remark}[Relation to reflection coupling of diffusion processes]
A reflection coupling of two copies of a diffusion process satisfying a stochastic differential equation of the form \eqref{eqSDE} is given by
\begin{equation}\label{eq:diffusionCoupling}
\begin{split}
\dd X_t &= b(X_t)  \dd t  +  \dd B_t,\\
dY_t &= b(Y_t)  \dd t  +  (I-2e_te_t^T)  \dd B_t\quad\mbox{for }t<T,
\quad X_t=Y_t\quad\mbox{for }t\ge T,
\end{split}
\end{equation}
where $e_t=\nicefrac{(X_t-Y_t)}{\lvert X_t-Y_t\rvert}$ and $T=\inf\{ t\ge 0:X_t=Y_t\}$ is the coupling time. Hence the noise increment is reflected up to the
coupling time, whereas after time $T$, $X_t$ and $Y_t$ move 
synchronously. Our coupling in discrete time has a similar effect. If
$\hat x$ and $\hat y$ are far apart then the transition densities
$\phi_{\hat x,hI} $ and $\phi_{\hat y,hI}$ have little overlap,
and hence reflection coupling is applied with very high probability.
If, on the other hand, $\hat x$ and $\hat y$ are sufficiently close, then with a non-negligible probability, $X'=Y'$. Once both Markov
chains have reached the same position, they stick together since their transition densities coincide subsequently. In this sense, the coupling \eqref{eq:0c} is a natural discretization
of reflection coupling. Indeed, we would expect that as $h\downarrow 0$, the coupled Markov chains with
time rescaled by a factor $h$ converge in law to the reflection coupling \eqref{eq:diffusionCoupling} of the diffusion processes. {On the other hand, a coupling of Markov chains in which jumps are always reflected (i.e., a coupling without the positive probability of jumping to the same point) would converge as $h\downarrow 0$ to a reflection coupling of diffusions in which the coupled processes do not follow the same path after the coupling time.}
\end{remark}

We assume that under the probability measure $\P_{x,y}$, $(X',Y')$ is the coupling of $p(x,\wc)$
and $p(y,\wc)$ introduced above. We set
\begin{equation}
\label{eqchoicer0}
r_0 \coloneqq \sqrt h ,
\end{equation}
and we consider the intervals
\begin{equation}
\label{eq:int}
I_r\ =\ \begin{cases}
(0,r+\sqrt h )&\mbox{for }r<r_0,\\
(r-\sqrt h, r) &\mbox{for }r\ge r_0.\end{cases}
\end{equation}
Thus in the notation from Section \ref{sec:secondresult}, we set
\begin{equation}
\label{equl}
u(r)\coloneqq \sqrt h  \1_{r<r_0},\qquad l(r) \coloneqq \sqrt h  \1_{r\ge r_0}.
\end{equation}
For given $x,y\in\r^d$ let $
r = \lvert x-y\rvert$, $ \hat r = \lvert \hat x-\hat y\rvert$, $R'\ =\ \lvert X'-Y'\rvert$,
\begin{eqnarray}
\label{eqhatbeta}
\hat\beta (x,y) &=& \E_{x,y}[R'-\hat r],\\
\label{eqhatalpha}
\hat\alpha (x,y) &=& \E_{x,y}\left[ \lvert((R'-\hat r)\wedge u(\hat r))\vee (-l(\hat r))\rvert^2\right],\qquad\mbox{and}\\
\label{eqhatpi}
\pi (x,y) &=& \P_{x,y}[R'=0].
\end{eqnarray}
In particular,
\begin{equation}\label{eqhatalpha2}
\hat\alpha (x,y) \ \ge\ \mathbb{E}_{x,y}[(R' - \hat{r})^2 \1_{\{ R' \in I_{\hat{r}} \}}] \,.
\end{equation}
Notice that the definitions of $\hat\beta$ and $\hat\alpha$ differ from those of $\beta$ and $\alpha$ given in \eqref{eq:18a} and \eqref{eq:19a},
since $\hat\beta$ and $\hat\alpha$ take into account only the coupled 
random walk transition step from $(\hat x,\hat y)$ to $(X',Y')$, but not 
the deterministic transition from $(x,y)$ to $(\hat x,\hat y)$. We also
consider 
\begin{eqnarray}
\label{eq:beta}
\beta (x,y)& =& \E_{x,y}[R'-r]\ =\ \hat\beta (x,y)+\hat r-r,\qquad\text{and}\\
\alpha (x,y)& =& \E_{x,y}\left[ \lvert((R'-r)\wedge u( r))\vee (-l( r))\rvert^2\right] .
\end{eqnarray}

\subsubsection*{Assumptions}
In our main result for the Euler scheme we assume that there exist 
constants $J\in [0,\infty )$ and $K,L,\mathcal R\in (0,\infty )$
such that the following conditions hold:
\begin{enumerate}[label=(C\arabic*)]
\item\label{enum:c1} {\em One-sided Lipschitz condition:}
\[
(x-y)\cdot (b(x)-b(y))\le J  \lvert x-y\rvert^2\qquad \mbox{for any }x,y\in\r^d.
\]
\item\label{enum:c2} {\em Strict contractivity outside a ball:}
\[
(x-y)\cdot (b(x)-b(y))\le -K  \lvert x-y\rvert^2\qquad \mbox{if }\lvert x-y\rvert\ge\mathcal R.
\]
\item\label{enum:c3} {\em Global Lipschitz condition:}
\[
\lvert b(x)-b(y)\rvert\le L  \lvert x-y\rvert\qquad \mbox{for any }x,y\in\r^d.
\]
\end{enumerate}
Notice that by \ref{enum:c2} and \ref{enum:c3}, $L\ge K$. Of course, \ref{enum:c3} implies \ref{enum:c1} with $J=L$. Note, however, that we can often choose $J$ much smaller than $L$, e.g., we can even
choose $J=0$ if $b=-\nabla U$ for a convex function $U\in C^2(\r^d)$.
The global Lipschitz condition is required for the stability of the Euler
scheme, but the constant $L$ will affect our lower bound for the contraction rate only in
a marginal way. On the other hand, our bound on the contraction rate will depend in an essential way on the one-sided Lipschitz constant
$J$.\medskip

  The bounds provided in the next lemma are crucial to apply the techniques developed above to the Euler scheme.

\begin{lemma}\label{lem:6}
Let $x,y\in\r^d$, and let $r=\lvert x-y\rvert$ and $\hat r=\lvert\hat x-\hat y\rvert$. Then
\begin{enumerate}
\item \label{enum:lem6i} $\hat\beta (x,y)= 0$,
\item \label{enum:lem6ii} $\hat\alpha (x,y)\ge c_0 \min (\hat r,\sqrt h)  \sqrt h,$
\qquad{and}
\item \label{enum:lem6iii} $\pi (x,y)\ge p_0  \1_{\hat r\le 2 \sqrt h}$,
\end{enumerate}
where $c_0,p_0\in (0,1 )$ are explicit universal constants 
($c_0\ge 0.007$, $p_0\ge 0.15$). Furthermore, if the assumptions
\ref{enum:c1}, \ref{enum:c2} and \ref{enum:c3} hold true, then
\begin{enumerate}[resume]
\item\label{enum:lem6iv} $\beta (x,y)\ \le\ \min (L,J+\nicefrac{L^2h}{2})  hr$,
\item\label{enum:lem6v} $\beta (x,y)\ \le\ -(K-\nicefrac{L^2h}{2})  hr$\qquad if $r\ge\mathcal R$,
\item\label{enum:lem6vi} $\alpha (x,y)\ \ge\ \tilde c_0h\1_{ r\ge \sqrt h}$\qquad if $ r\le
1/(4L\sqrt h)$,\qquad and
\item\label{enum:lem6vii} $\pi (x,y)\ \ge\ p_0 \1_{ r\le \sqrt h}$\qquad if $h\le 1/L$.
\end{enumerate}
Here $\tilde c_0$ is an explicit universal constant 
($\tilde c_0\ge 0.0005$).
\end{lemma}
The proof of the lemma is contained in Section \ref{sec:proofsEuler}.

\subsubsection*{Contractive case}
At first, we consider the case where the deterministic part of the Euler
transition is a contraction, i.e.,
\begin{equation}
\label{eq:30*}
\hat r=\lvert\hat x-\hat y\rvert\le\lvert x-y\rvert=r\qquad\mbox{for any }x,y\in\r^d.
\end{equation}
In this simple case, we can prove a rather sharp result.
We choose a metric $\rho_a\colon \r^d\times\r^d\to\r_+$ of type
\begin{equation}
\label{eq:31}
\rho_a (x,y)\ =\ a\1_{x\neq y}\, +\, f_a(|x-y|),\qquad f_a(r)\ =\ \int_0^rg_a(s\wedge\mathcal R ) \dd s.
\end{equation}
Here $a$ is a non-negative constant, $\mathcal R$ is chosen as in Assumption \ref{enum:c2}, and $g_a\colon [0,\mathcal R]\to\r$ is an appropriately chosen decreasing function (see \eqref{eq:7**f} for $a=0$ and \eqref{eq:7**f2} for $a\neq 0$) satisfying
\begin{equation}
\label{eq:32}
g_a(0)=1 \qquad\mbox{and}\qquad g_a(s)\  \in\ [\nicefrac{1}{2},1]\quad
\mbox{for any }s\in [0,\mathcal R ].
\end{equation}
Hence $f_a$ is a concave increasing function satisfying $\nicefrac{r}{2}\le f_a(r)\le r$, and thus
\begin{equation}
\label{eq:33}
a\1_{x\neq y}+\nicefrac{\lvert x-y\rvert}{2}\le\rho_a (x,y)\le a\1_{x\neq y}+\lvert x-y\rvert\qquad\mbox{for any }x,y\in\r^d.
\end{equation}
In particular, the distance $\rho_0$ is equivalent to the Euclidean distance.

\begin{thm}[Euler scheme, contractive case]\label{thm:7}
Suppose that Conditions \ref{enum:c1}, \ref{enum:c2}, \ref{enum:c3} and \eqref{eq:30*} are
satisfied, and let $h_0=\frac 1L\min \left(\frac{K}{L},\frac 12\right)$. Suppose that $a=0$ or $a\ge \sqrt h$, and let 
$\rho_a$ be defined by \eqref{eq:31} with $g_a$ 
specified in \eqref{eq:7**f},  \eqref{eq:7**f2}, respectively. Let 
\begin{eqnarray}
\label{eq:c1h0}
c_1(0)& =& \frac{1}{4}\min  \left( K,\,   \frac{2c_0  }{\mathcal R^2+2\sqrt h\mathcal R+
12\, h}\right) \qquad\text{ and}\\
\label{eq:c1h0a}
c_1(a)& =& \frac{1}{4}\min  \left( \frac{K}{1+a/\mathcal R},\,   \frac{2c_0  }{\mathcal R^2+2(a+\sqrt h)\mathcal R},\, \frac{2p_0}{h}\right)\quad\text{for }a>0,
\end{eqnarray}
where $c_0$ is the explicit constant in Lemma \ref{lem:6}. Then
\eqref{eq:32} and \eqref{eq:33} hold, and if $h\in (0,h_0)$, then
\begin{equation}
\label{eq:34}
\E_{x,y}[\rho_a (X',Y')]\le\left(1-c_1(a)h\right) \rho_a (x,y)\quad
\mbox{for all }x,y\in\r^d.
\end{equation}
\end{thm}

The proof, based on Theorem \ref{thm:1} for $a > 0$ and
Theorem \ref{thm:3} for $a = 0$, is given in Section \ref{sec:proofsEuler}. 

\begin{remark}[Dependence on parameters and dimension]\label{remark:thm1}
The lower bound for the contraction rate in
\eqref{eq:34} is of the correct order $\Omega (h\min (\mathcal R^{-2},K))$. This corresponds to the optimal contraction rate $\Theta (\min (\mathcal R^{-2},K))$ for the corresponding diffusion process, see \cite[Lemma 1 and Remark 5]{Eberle2015}. Note also that the lower bound for the contraction rate does not depend on the dimension $d$ provided the parameters $\mathcal R,K$ and $L$ can be chosen independent of $d$.
\end{remark}

\subsubsection*{General case}
We now turn to the general, not globally contractive case. 
Here it is no longer possible
to obtain contractivity w.r.t.\ a metric satisfying \eqref{eq:33}, but we can still choose a metric that is comparable to the Euclidean distance, and apply the theorems above. We illustrate this at first by applying Theorem \ref{thm:1}. Let
\begin{equation}
\label{Lambda} \Lambda \ =\  \min(L, J+L^2h/2).
\end{equation}
We now choose a metric $\rho_a\colon \r^d\times\r^d\to\r_+$ of type
\begin{equation}
\label{eq:31x}
\rho_a (x,y)\ =\ a\1_{x\neq y}\, +\, f_a(|x-y|),\qquad f_a(r)\ =\ \int_0^rg_a(s\wedge r_2 )\varphi (s\wedge r_2 ) \dd s.
\end{equation}
Here $a$ is a non-negative constant, 
\begin{equation}
\label{eq:31y}
\varphi (r)\ =\ \exp \left( -\tilde c_0^{-1}\Lambda \left( (r\wedge\mathcal R)^2+2\sqrt h\,  r\wedge \mathcal R\right)\right) 
\end{equation}
with $\mathcal R$ and $\tilde c_0$ chosen as in Assumption \ref{enum:c2} and Lemma \ref{lem:6}, respectively,
\begin{equation}
\label{eq:31ya}
r_2\ =\ \mathcal R +\sqrt{2\tilde c_0/K},
\end{equation}
and $g_a\colon [0,r_2 ]\to\r$ is an appropriately chosen decreasing function (see \eqref{eq:7**f2}) satisfying
\begin{equation}
\label{eq:32x}
g_a(0)=1 \qquad\mbox{and}\qquad g_a(s)\  \in\ [\nicefrac{1}{2},1]\quad
\mbox{for any }s\in [0,r_2 ].
\end{equation}

\begin{thm}[Euler scheme, general case I]\label{thm:8a}
Suppose that Conditions \ref{enum:c1}, \ref{enum:c2} and \ref{enum:c3} are satisfied, and let $h_0 = \frac{1}{L}\min (\frac{p_0}{2}, \frac{K}{L}, \frac{1}{64\, L  r_2^2})$ with $r_2$
specified in \eqref{eq:31ya}. Let $a \in [2\sqrt{h}, \Phi (\mathcal R)]$ where $\Phi (\mathcal R ):=\int_0^\mathcal R\varphi (r)\, dr$, let $\rho_a$ be defined by \eqref{eq:31x} with $\varphi$ and $g_a$ 
specified in \eqref{eq:31y} and \eqref{eq:7**f2}, respectively, and let 
\begin{equation}\label{eq:34x}
c_2(a)\ =\ \frac{1}{8}\min  \left( \frac{K\, \varphi (\mathcal R )}{1+(a+/\sqrt h)\sqrt{2K/\tilde c_0}},\,   \frac{2\, \tilde c_0\, \varphi (\mathcal R )  }{\mathcal R^2+2(a+\sqrt h)\mathcal R},\, \frac{4p_0}{h}\right) .
\end{equation}
Then \eqref{eq:32x} holds, and if $h \in (0, h_0)$, then 
\begin{equation*}
\E_{x,y}[\rho_a (X',Y')]\le\left(1-c_2(a)h\right) \rho_a (x,y)\quad
\mbox{for all }x,y\in\r^d.
\end{equation*}
\end{thm}	

Note that except for the
additional factor $\varphi (\mathcal R)=\exp ( -\tilde c_0^{-1}\Lambda (\mathcal R^2+2\sqrt h\mathcal R))$, the expression for the contraction rate $c_2(a)$
is similar to the one for the rate $c_1(a)$ in the contractive case. 
The proof based on Theorem \ref{thm:1} is given in Section \ref{sec:proofsEuler}. If the interval $[2\sqrt{h}, \Phi (\mathcal R)]$ is empty, the theorem can still be applied with $\mathcal R$ replaced by a slightly larger value. 
It is also possible to replace Condition \ref{enum:c2} by a Lyapunov condition and apply Theorem \ref{thm:2} instead of Theorem \ref{thm:1}. A corresponding result for the Euler scheme is
given in Section \ref{sec:proofsEuler}, cf.\ Theorem \ref{thm:8b}.

\begin{remark}[Dependence on parameters and dimension]\label{remark:thm2}
The lower bound for the contraction rate in
\eqref{eq:34x} does not depend on the dimension $d$ provided the parameters $\mathcal R,K$ and $\Lambda$ can be chosen independent of $d$. Moreover, by choosing $h$ sufficiently small, we can ensure that $\Lambda$ is close to the one-sided Lipschitz constant $J$. Hence the global Lipschitz constant $L$ is only required for controlling the step size $h$, whereas the contraction properties for sufficiently small $h$ can be controlled essentially by one-sided Lipschitz bounds. This is important since in many applications, only a one-sided Lipschitz condition is satisfied globally. In this case, our approach can still be
applied on a large ball if the step size is chosen sufficiently small depending on the radius of the ball and the growth of the local Lipschitz constant.
\end{remark}

The explicit expression for the metric in Theorem \ref{thm:8a} is a bit complicated. As an alternative, we can use 
a simplified metric without a discontinuity that is sufficient to derive bounds of similar order as for the 
metric used above, whenever condition \ref{enum:c2} is satisfied. 
We assume $hL\le 1/6$,
and we set 
\begin{equation}
\label{eq:35}
r_1\ :=\ (1+hL)\mathcal R\ \le\ \frac 76\mathcal R
\end{equation}
with $\mathcal R$ and $L$ as in Assumptions \ref{enum:c2} and \ref{enum:c3}. The choice of $r_1$ ensures that
\begin{equation}
\label{eq:36}
\hat r\ =\ |\hat x-\hat y|\ \le\ (1+hL)r\ \le\ r_1\qquad\text{whenever }r=|x-y|\le\mathcal R.
\end{equation}
Let $c_0$ denote the explicit constant in Lemma \ref{lem:6}, and let
\begin{equation}
\label{eq:37}
q\ =\ 7c_0^{-1}\Lambda\mathcal R.
\end{equation}
We now consider a simplified metric of the form
\begin{align}
\label{eq:39}
\rho (x,y)&=f(\lvert x-y\rvert),& f(r)&=\int_0^r\exp (-q(s\wedge r_1))
 \,d s.
\end{align}

\begin{thm}[Euler scheme, general case II]\label{thm:8}
Suppose that Conditions \ref{enum:c1}, \ref{enum:c2} and \ref{enum:c3} are
satisfied, and let $\rho$ be defined by \eqref{eq:39} with $q$
specified in \eqref{eq:37}. Let 
\begin{eqnarray}\label{eq:c2}
c_2& =&  \min  \left( \frac K2,\, \frac{245}{24 c_0} {\Lambda^2\mathcal R^2} \right)\, \exp\left({-\frac{49}{6 c_0}{\Lambda \mathcal R^2}}\right) \qquad \text{ and }\\ 
h_0&=&\frac 1L\min \left(\frac 16,\, \frac KL,\, \frac 13{L\mathcal R^2},\,\frac{c_0^2}{970}\frac 1{L\mathcal R^2}\right), \label{eq:h0general}
\end{eqnarray}
where $c_0$ is chosen as in Lemma \ref{lem:6}. Then 
\begin{equation}
\label{eq:45}
\E_{x,y}[\rho (X',Y')]\le(1-c_2h) \rho (x,y)\quad
\mbox{for any }x,y\in\r^d\mbox{ and }h\in (0,h_0].
\end{equation}
\end{thm}
The proof of the theorem is contained in Section \ref{sec:proofsEuler}. 
 
\begin{remark}\label{remark:thm3}
Again, the lower bound $c_2$ for the contraction rate only depends on
$\mathcal R$, $K$ and $\Lambda$. Furthermore, note that $r \exp (-qr_1) \leq f(r) \leq r$ for all $r \geq 0$, and hence the metric $\rho$ is comparable to the Euclidean distance. As a consequence, Theorem \ref{thm:8} implies weak contractivity in the standard $L^1$ Wasserstein distance. Note also that the function $f$ depends on the discretization parameter $h$ via $q$ and $r_1$. It is, however, possible to modify the definition of $f$ so that it no longer depends on $h$, at the cost of getting a worse constant $c_2$. We refer the interested reader to \cite{MajkaMijatovicSzpruch}, where similar bounds are used with a metric independent of $h$. 
\end{remark}

Theorem \ref{thm:8} can be extended to cover pseudo metrics based on functions that are strictly convex at infinity. This allows for obtaining upper $L^2$ bounds for Euler schemes under similar assumptions as above. Such bounds are applied to the analysis of Multi-level Monte Carlo algorithms in the upcoming paper \cite{MajkaMijatovicSzpruch}.

\subsection{Application to MALA}\label{sec:\MALA}

The Metropolis-adjusted Langevin Algorithm is a 
Metropolis-Hastings method for approximate sampling from a given
probability measure $\mu$ where the proposals are obtained by
an Euler discretization of an overdamped Langevin SDE. In
\cite{EberleAAP}, the dimension dependence of contraction rates
of MALA chains w.r.t.\ standard Kantorovich distances has been
studied for a class of strictly log-concave probability measures that
have a density w.r.t.\ a Gaussian reference measure. Our goal is
to provide a partial extension of these results to non log-concave 
measures. By considering the MALA transition step as a 
perturbation of the Euler proposals, we obtain contraction rates
w.r.t.\ a modified Kantorovich distance provided the discretization
time step is of order $h=O(d^{-1})$.\medskip

We consider a similar setup as in \cite{EberleAAP}: $\mu$ is a 
probability measure on $\mathbb R^d$ given by 
\begin{equation}
\label{eq:m1}
\mu (d x)=\mathcal Z^{-1} \exp (-U(x))  \dd x=
(2\pi )^{\nicefrac{d}{2}} \mathcal Z^{-1}  \exp (-V(x))  \gamma^d(d x) \,,
\end{equation}
where $V$ is a function in $C^4(\mathbb R^d)$,
\begin{equation}
\label{eq:m1a}
U(x)=\frac 12\lvert x\rvert^2  +  V(x),
\end{equation}
$\gamma^d$ denotes the $d$-dimensional standard normal
distribution, and 
\[
\mathcal Z=\int\exp (-U(x))  \dd x.
\]
We assume
that we are given a norm $\|\cdot\|_-$ on $\mathbb R^d$ such that
\begin{equation}
\label{eq:m2}
\| x\|_-\ \le\ |x|\ \le\ d  \| x\|_-\qquad\mbox{for any }x\in \mathbb R^d,
\end{equation}
as well as finite constants $C_n\in [0,\infty )$, $p_n\in \{ 0,1,2,\ldots
\}$, and $K_c,\mathcal R_c\in (0,\infty )$ such that the following conditions hold for any $n \in \{1, \dots, 4\}$:
\begin{eqnarray}
\label{eq:d1}  \quad\lvert \partial_{\xi_1, \dots, \xi_n}^n U(x) \rvert &\le& C_n \max (1, \| x \|_-^{p_n}) \| \xi_1 \|_- \cdots \| \xi_n \|_-\, \forall\, x, \xi_1, \dots, \xi_n \in \R^d.  \\
\label{eq:d2}  (\partial_{\xi \xi} U)(x) &\ge & K_c \lvert \xi \rvert^2  \qquad\forall\ x, \xi \in \R^d: \lvert x \rvert \ge \mathcal R_c  .
\end{eqnarray}
Here \eqref{eq:d2} can be interpreted as strict convexity of $U$ outside a Euclidean ball. 
\begin{remark}
\begin{enumerate}
\item 
For discretizations of infinite-dimensional models, $\| \wc \|_-$ is typically a finite-dimensional approximation of a norm that is almost surely finite w.r.t. the limit measure in infinite dimensions, see for instance \cite[Example 1.6]{EberleAAP}. Correspondingly, we may assume that the measure concentrates on a ball of a fixed radius w.r.t. $\|\wc\|_-$. This will be relevant for the application of Theorem \ref{thm:10} below, which states uniform contractivity on such balls. 
\item 
Condition \eqref{eq:d1} is the same condition that has been assumed in the strictly convex case in \cite{EberleAAP}. 
\item 
In \eqref{eq:d2}, we assume strict convexity outside a ball of fixed radius w.r.t.\ the Euclidean norm and not w.r.t.\ $\| \wc \|_-$. Such a bound can be expected to hold with $\mathcal R_c$ independent of the dimension if, for example, the non-convexity occurs only in finitely many directions. The application of a coupling approach in situations where \eqref{eq:d2} does not hold requires more advanced techniques, see e.g.\! \cite{Zimmer}.
\end{enumerate}
\end{remark}
The transition step of a Metropolis-Hastings chain with proposal density $p(x,y)$ and target distribution $\mu(d x) = \mu(x) \dd x$ is given by
\begin{equation}
\label{eq:m4} \tilde{x} = \begin{cases}X' &\text{if } \tilde{U} \le \alpha(x,X') \\ x & \text{otherwise.} \end{cases},
\end{equation}
where $x$ is the previous position, $X'$ is the proposed move, 
\[
\alpha (x,y) = \min \left(1, \frac{\mu(y) p(y,x) }{\mu(x) p(x,y)}\right)
\]
is the Metropolis-Hastings acceptance probability, and $\tilde{U} \sim \Unif (0,1)$ is a uniform random variable that is independent of $X'$. We consider the proposal
\begin{equation}
\label{eq:m6} X' = x-\frac{h}{2} x - \frac{h}{2} \nabla V(x) + \sqrt{h - \frac{h^2}{4}} Z, \quad Z \sim N(0, I_d) \,,
\end{equation}
where $h \in (0,2)$ is the step size of the time discretization. The corresponding proposal kernel is $p_h(x,\wc) = N(x-\frac{h}{2} x - \frac{h}{2} \nabla  V(x), (h - \nicefrac{h^2}{4}) I_d)$. Substituting $h = \frac{\varepsilon}{1+\nicefrac{\varepsilon}{4}}$, we see that the proposal is a transition step of the semi-implicit Euler discretization
\begin{equation}
\label{eq:m7} X' = x - \frac{\varepsilon}{2} \frac{X'+x}{2} - \frac{\varepsilon}{2} \nabla V(x) + \sqrt{\varepsilon} Z
\end{equation}
for the Langevin SDE $\dd X_t = -\frac{1}{2} X_t\dd t - \frac{1}{2} \nabla V(X_t) \dd t + \dd B_t$ with invariant measure $\mu$. The reason for considering the semi-implicit instead of the explicit Euler approximation is that under appropriate conditions, the acceptance probability
\begin{equation}
\label{eq:m8} \alpha_h (x,y) = \min \left(1, \frac{\mu(y) p_h(y,x)}{\mu(x) p_h(x,y)}\right)
\end{equation}
for the corresponding Metropolis-Hastings scheme has a better dimension dependence. Indeed, if $V$ vanishes, then $\alpha_h(x,y) = 1$. More generally, if \eqref{eq:d1} holds, then the average rejection probability is of order $O(h^{\frac{3}{2}})$.
\begin{lemma}[Upper bounds for rejection probability]\label{lem:3}
Suppose that \eqref{eq:d1} holds and let $k \in \N$. Then there exists an explicit polynomial $P_k \colon \R^2 \rightarrow \R_+$ of degree $\max(p_3 + 3, 3p_2 + 2)$ such that for any $x \in \R^2$ and $h \in (0,2)$,
\[
\E[ (1-\alpha_h(x,X'))^k]^{\frac{1}{k}} \le P_k (\| x \|_-, \| x + \nabla V(x) \|_-)h^{\frac{3}{2}} \,.
\]
\end{lemma}
The proof of the lemma is given in \cite[Proposition 1.7]{EberleAAP}. The polynomials $P_k$ are explicit. Their coefficients depend only on the constants $C_2, C_3, p_2$ and $p_3$ in \eqref{eq:d1} and on the moments
\[
m_n = \E[\| Z\|_-^n], \, n \le k \max (p_3 + 3, 2p_2 + 2) \,.
\]
Apart from replacing $\sqrt{h}$ by $\sqrt{h - \frac{h^2}{4}}$, \eqref{eq:m6} coincides with the explicit Euler discretization of the SDE $\dd X_t = b(X_t) \dd t + \dd B_t$, where $b(x) = -\frac{1}{2} x - \frac{1}{2} \nabla V(x)$. 

Therefore, the results in the last section apply to the proposal chain, thus yielding a contraction rate of order $\Omega (h)$. Since the rejection probability is of higher order, we can then apply the perturbation result in \eqref{thm:4} to prove a corresponding contractivity for the \MALA chain. To this end, we consider the coupling $(\tilde{X}, \tilde{Y})$ of transition steps of the \MALA chain from positions $x,y \in \R^d$ given by \eqref{eq:m4} and
\begin{equation}
\label{eq:m9} \tilde{Y} = \begin{cases}
 Y' & \text{if } \tilde{U} \le \alpha(y,Y') \\ y &\text{otherwise.}
 \end{cases}
\end{equation}
where $(X', Y')$ is the (optimal) coupling for the proposal steps considered in (\ref{eq:0c}), and $\tilde{U} \sim \Unif(0,1)$ is independent of both $X'$ and $Y'$. Hence, the proposals are coupled optimally and the same uniform random variable $\tilde{U}$ is used to decide about acceptance or rejection for each of the steps. Nevertheless, in general $(\tilde{X}, \tilde{Y})$ is not an optimal coupling of the corresponding \MALA transition probabilities. 

\begin{theorem}[Contraction rates for MALA] \label{thm:10}
Suppose that conditions \eqref{eq:d1} and \eqref{eq:d2} hold and fix $ R \in (0,\infty)$. Then there exists a concave strictly increasing function $\tilde f \colon [0,\infty) \rightarrow [0,\infty)$ with $\tilde f(0) = 0$ and constants $c_3, h_0 \in (0,\infty)$ such that 
for any $h \in (0,h_0d^{-1})$ and for any $x,y \in \R^d$ with $\| x \|_- \le R$, $\|y\|_- \le R$, 
\begin{equation}
\label{eq:m10} \E_{x,y} [\tilde f(|\tilde{X}- \tilde{Y}|) ] \le (1-c_3 h) \tilde f(|x-y|) \,.
\end{equation}
The function $\tilde f$ and the constants $c_3$ and $h_0$ depend only on $R$ and on the values of the constants $C_n$, $p_n$, $K_c$, $\mathcal R_c$ in assumptions \eqref{eq:d1}, \eqref{eq:d2}. 
\end{theorem}
The proof of Theorem \ref{thm:10} is given in Section \ref{sec:proofof\MALA}.
\begin{remark}
The theorem shows that by choosing the step size of order $\Theta (d)$, a contraction rate of the same order holds on balls w.r.t. $\| \wc \|_-$ provided conditions \eqref{eq:d1} and \eqref{eq:d2} are satisfied. In the strictly convex case, it has been shown in \cite{EberleAAP} by a synchronous coupling that a corresponding result holds even for step sizes of order $\Theta (1)$ if the Euclidean norm in \eqref{eq:m10} is replaced by $\| \wc \|_-$. One could hope for a similar result in the not globally convex case, but the combination of reflection coupling with a different norm leads to further difficulties. A possibility to overcome these difficulties might be the two-scale approach developed in \cite{Zimmer}. 
\end{remark}

\section{Proofs of Theorems \ref{thm:1} and \ref{thm:2}}\label{sec:proofs12}
In this section, we prove the first two theorems. We first specify the explicit choice of the metric and the explicit values of the contraction rate $c$. The reason for choosing the metric this way will become clear by the subsequent proofs of the theorems. 

For $r, s > 0$, we consider the intervals
\begin{equation}
\label{eq:501} I_r = ((r-\varepsilon)^+,r)
\end{equation}
and the dual intervals
\begin{equation}
\label{eq:502} \hat{I}_s = \{ r > r_0 : s \in I_r\} = (s,s+\varepsilon) \cap (r_0, \infty) \,.
\end{equation}
For $r \in (r_0,\infty)$ we set
\begin{equation}
\label{eq:5stara} \overline{\gamma} (r) = \nicefrac{2\overline{\beta} (r)}{\underline{\alpha} (r)} \,.
\end{equation}
Let $\tilde{\gamma} \colon [0,\infty) \rightarrow [0,\infty)$ be a function satisfying
\begin{eqnarray}
\label{eq:5star} \sup_{r \in \hat{I}_s} \overline{\gamma} (r) &\le & \tilde{\gamma} (s) \quad\text{ for any } s \in [0,\infty),\qquad\text{i.e.,}\\
\label{eq:503} \overline{\gamma}(r) &\le & \tilde{\gamma} (s) \tforall r > r_0, \, s \in I_r \,.
\end{eqnarray}
By assumptions \ref{enum:a1} and \ref{enum:a2}, such a function exists. If \ref{enum:a3} holds, then we may assume w.l.o.g. that $\tilde{\gamma} (s) = 0$ for large $s$.

\subsection{Choice of the metric in Theorem \ref{thm:1}}\label{subsec:choiceOfMetricThm1}
Suppose conditions \ref{enum:a1}, \ref{enum:a2} and \ref{enum:a3} hold. We set
\begin{equation}\label{eq:choice r1}
r_1 \coloneqq \sup \{ r > 0 : \tilde{\gamma} (r) > 0\},
\end{equation}
where $\sup \emptyset = 0$. By Assumption \ref{enum:a3} we can choose $\tilde{\gamma}$ such that $r_1$ is finite. We have 
\begin{align}
\label{eq:50} \tilde{\gamma}(r) &= 0 \,,& \overline{\gamma} (r) &\le 0 \,,&\overline{\beta}(r) &\le 0 \, \tforall r \ge r_1 \,.
\end{align}
We also fix a constant $r_2 \in (r_1, \infty)$. The value of $r_2$ will be determined in condition \eqref{eq:5d} below. The underlying metric we consider is given by \eqref{eq:13}, where $f \colon [0,\infty) \rightarrow [0,\infty)$ is a concave increasing function defined by 
\begin{equation}
\label{eq:5000} f(r) = a \1_{r > 0} + \int_0^r \varphi (s \wedge r_2) g(s\wedge r_2) \dd s
\end{equation}
with decreasing differentiable functions $\varphi$ and $g$ such that $\varphi(0) = g(0) = 1$ and a constant $a \in (0,\infty)$ that are all specified below. Hence, $f$ is twice differentiable except at $0$, $f(0+) -f(0) = a$, $f' = \varphi g$ on $(0,r_2)$, and $f'$ is constant on $[r_2, \infty)$.

The function $\varphi$ and the constant $a$ are chosen such that
\begin{align}
\label{eq:5a} \varphi(r) &= \exp\left(-\int_0^r \tilde{\gamma} (s) \dd s \right),\qquad\text{and}\\
\label{eq:5c} a &\ge r_0 + 2 \sup_{\lvert x-y \rvert \le r_0} \frac{\beta(x,y)}{\pi(x,y)} .
\end{align}
Notice that by \eqref{eq:50}, the function $\varphi (r)$ is constant for $r\ge r_1$. 
Setting
\begin{equation}
\label{eq:5e} \Phi(r) = \int_0^r \varphi(s) \dd s \,,
\end{equation}
the constant $r_2$ is chosen such that
\begin{equation}
\label{eq:5d}-\frac{\overline{\beta(r)}}{a+\Phi(r)} \ge \frac{1}{2} \left(\int_{r_1}^{r_2} \frac{\Phi(s)}{\underline{\alpha}(s)} \dd s \right)^{-1} \, \text{ for all } r \ge r_2 \,.
\end{equation}
Assumption \ref{enum:a3} ensures that such a constant exists. Indeed, for $r \ge r_1$, $\tilde{\gamma}$ vanishes, whence $\varphi$ is constant and $\Phi$ is linear. By definition of $\alpha$, {we see that} $\underline{\alpha}$ is uniformly bounded by $\varepsilon^2$. Therefore, the value on the right hand side of \eqref{eq:5d} goes to zero as $r_2 \rightarrow \infty$, and \eqref{eq:5d} holds for large $r_2$ by \ref{enum:a3}. 

The contraction rate is now given by 
\begin{equation}
\label{eq:5f} c = \min\left( \frac{1}{2} \inf_{r \leq r_0} \underline{\pi}(r), \frac{1}{4} \left( \int_0^{r_2} \frac{1}{\varphi(s)} \sup_{u \in \hat{I}_s} \frac{a + \Phi (u)}{\underline{\alpha} (u)} \dd s \right)^{-1}\right)
\end{equation}
and the function $g$ is defined as 
\begin{equation}
\label{eq:5g} g(r) = 1-2c \int_0^r \frac{1}{\varphi(s)} \sup_{u \in \hat{I}_s} \frac{a + \Phi (u)}{\underline{\alpha} (u)} \dd s \,.
\end{equation}
Note that \eqref{eq:5f} guarantees that $g(r) \ge \frac{1}{2}$ for $r \le r_2$.

\subsection{Choice of the metric in Theorem \ref{thm:2}}\label{subsec:choiceOfMetricThm2}
Now suppose that \ref{enum:a1}, \ref{enum:a2} and \ref{enum:a4} hold. In this case we set
\begin{equation}
\label{eq:6r1} r_1 \coloneqq \sup \{ d(x,y) : x,y \in S \,, \, V(x) + V(y) < \nicefrac{4C}{\lambda}\} \,.
\end{equation}
By \ref{enum:a4ii} and \ref{enum:a2}, $r_1$ is finite. Moreover, by \ref{enum:a4i},
\begin{equation}
\label{eq:60} \E_{x,y} [V(X') + V(Y') ] \le \left(1-\frac{\lambda}{2}\right) (V(x) + V(y) )\qquad \text{if } d(x,y) \ge r_1 \,.
\end{equation}
We also fix a constant $r_2 \in (r_1, \infty)$. The value of $r_2$ will be determined by condition \eqref{eq:6d} below. 

The function $f \colon [0,\infty) \rightarrow [0,\infty)$ determining the metric in \eqref{eq:15} is now defined by 
\begin{equation}
\label{eq:6000} f(r) = a \1_{r>0} + \int_0^r \varphi(s \wedge r_2) g(s\wedge r_2) \dd s
\end{equation}
with decreasing differentiable functions $\varphi$ and $g$ such that $\varphi(0) = g(0) = 1$, and a constant $a \in (0,\infty)$ that are all specified below. Hence, $f$ is twice differentiable except at $0$, $f(0+) - f(0) = a$, $f' = \varphi g$ on $(0,r_2)$ and $f'$ is constant on $[r_2, \infty)$.

The function $\varphi$ and the constants $a$ and $M$ in \eqref{eq:15} are chosen such that
\begin{align}
\label{eq:6a} \varphi(r) &= \exp \left(-\int_0^r \tilde{\gamma} (s) \dd s \right) \,, \\
\label{eq:6b} M &\le \frac{1}{4} \left(\int_0^{r_1} \frac{1}{\varphi(s) } \sup_{u \in \hat{I}_s} \frac{1}{\underline{\alpha} (u)} \dd s \right)^{-1} \,, \\
\label{eq:6c} a &\ge r_0 + 2 \sup_{\lvert x-y \rvert \le r_0} \frac{\beta(x,y) + M}{\pi (x,y)} \,, \\
\label{eq:6d}  \overline{\beta}(r) \varphi(r) &\le \frac{\lambda M}{16\, C} (V(x) + V(y))\ \text{ if } \lvert x-y \rvert \ge r_2 \,.
\end{align}
By \ref{enum:a4ii} and since $\varphi \le 1$, there always exists a finite $r_2$ such that \eqref{eq:6d} holds. To optimize the estimates, we choose $r_2$ as small as possible, i.e., we set
\begin{equation}
\label{eq:6e} r_2 = r_1 \vee \sup \left\{ d(x,y) = r : x,y \in S, {V(x) + V(y)} < \frac{16\, C}{\lambda M} {\overline{\beta} (r)}\varphi(r) \right\} \,.
\end{equation}
Setting
\begin{equation}
\Phi (r) = \int_0^r \varphi(s) \dd s \,,
\end{equation}
the contraction rate $c$ is given by
\begin{equation}
\label{eq:6f} c\ =\ \min \left( \frac{1}{2} \inf_{r \leq r_0} \underline{\pi}(r) ,\, \frac{\lambda}{4} ,\, \frac{\lambda M}{16\, C}\inf_{r\ge r_2}\frac{V(x)+V(y)}{\Phi (r)} ,\, \frac{1}{8} \left( \int_0^{r_2} \frac{1}{\varphi(s)} \sup_{u \in \hat{I}_s} \frac{a + \Phi (u)}{\underline{\alpha} (u)} \dd s \right)^{-1}\right) \,.
\end{equation}
and the function $g$ is defined as 
\begin{equation}
\label{eq:6g} g(r) = 1-2c \int_0^r \frac{1}{\varphi(s)} \sup_{u \in \hat{I}_s} \frac{a + \Phi (u)}{\underline{\alpha} (u)} \dd s - M \int_0^{r \wedge r_1} \frac{1}{\varphi(s)} \sup_{u \in \hat{I}_s} \frac{1}{\underline{\alpha} (u)} \dd s \,. 
\end{equation}
Note that \eqref{eq:6f} and \eqref{eq:6b} guarantee that $g(r) \ge \frac{1}{2}$ for $r \le r_2$. In the minimum defining $c$, the first term guarantees contractivity for $r \le r_0$, the second term is used for all $r$, the third term guarantees contractivity for $r \ge r_2$ and the last term ensures contractivity with rate $c$ for $r_0 < r \le r_1$.

\subsection{Proof of Theorem \ref{thm:1} and Theorem \ref{thm:2}}
Since the arguments are similar, we prove both theorems simultaneously, distinguishing cases where needed. In the situation of Theorem \ref{thm:1}, we set {$M=0$}. Let $x,y \in \R^d$ such that $r = \lvert x-y \rvert > 0$. Since $f''(t) \le 0$ for all $t>0$,
\begin{eqnarray*}
f(R')-f( r)&=&-a \1_{R' = 0}\, +\, \int_{ r}^{R'}f'(s)\,ds\\ 
&=& -a \1_{R' = 0}\, +\, (R'- r)f'( r)\, +\, \int_{ r}^{R'}\int_{ r}^sf''(t)\, dt\, ds\\
&\le & -a \1_{R' = 0}\, +\, (R' - r) f' (r)\, +\, \frac{1}{2} ((R'-r)^- \wedge \varepsilon)^2 \sup_{u \in {I}_r} f''(u),
\end{eqnarray*}
where $I_r = ((r-\varepsilon)^+, r)$. By taking expectations, we conclude that
\begin{equation}
\label{eq:7*}\E_{x,y} [f(R' ) -f(r)] \le -a \pi (x,y) + \beta (x,y) f'(r) + \frac{1}{2} \alpha(x,y)\sup_{u \in {I}_r} f''(u).
\end{equation}
Our goal is to compensate the second term by the first term for $r \le r_0$ and by the last term for $r_0 < r \le r_2$ (and possibly by a Lyapunov part for $r \ge r_2$). In order to verify \eqref{eq:14} and \eqref{eq:16}, we now distinguish three cases.\smallskip\\
\emph{Case $r \in (r_0, r_2)$.} Since $f' = g\varphi$ on $(0,r_2)$, we have 
\begin{equation}
\label{eq:71} \sup_{ {I}_r} f''\ \le\ \sup_{ {I}_r} (g' \varphi)  + \sup_{ {I}_r} (g\varphi') \,.
\end{equation}
Note that both summands are negative since $g$ and $\varphi$ are decreasing.
Now we note first that our choice of $\varphi$ guarantees that 
\begin{equation}
\label{eq:72} \frac{1}{2} \alpha(x,y) \sup_{ {I}_r} (g\varphi')  + \beta(x,y) f' (r) \le 0 \,.
\end{equation}
Indeed, \eqref{eq:72} is satisfied provided
\begin{equation}
\label{eq:72'} \sup_{ {I}_r} (g\varphi') \le - \overline{\gamma}(r) g(r) \varphi(r) \,.
\end{equation}
Since $\varphi' \le 0$ and $g$ is decreasing, we have
\[
\sup_{ {I}_r} (g\varphi') \le \inf_{ {I}_r} g\, \sup_{ {I}_r}\varphi'  \le g(r) \sup_{ {I}_r}\varphi' .
\]
Hence, \eqref{eq:72'} is satisfied if
\begin{equation}
\label{eq:72''}  \sup_{s \in {I}_r}\varphi'(s) \le -\overline{\gamma} (r) \varphi(r).
\end{equation}
But indeed, by definition of $\varphi$ and $\tilde{\gamma}$, we have for $s \in I_r$
\[
\varphi'(s) = -\tilde{\gamma} (s) \varphi(s) \le -\overline{\gamma} (r) \varphi(r) \,.
\]
Next, we observe that our choice of $g$ (in particular, $g \leq 1$) guarantees that
\begin{equation}
\label{eq:73} \frac{1}{2} \alpha(x,y) \sup_{{I}_r}(g'\varphi )  + M \1_{r < r_1} \le -c f(r) \,.
\end{equation}
Indeed, since $f(r) \le a + \Phi (r)$ and $f'(r) \le 1$, it is sufficient to show
\begin{equation}
\label{eq:73'} \sup_{s \in {I}_r}(g'(s)\varphi(s)) \le -2c \frac{a+\Phi (r)}{\underline{\alpha} (r)} - \frac{M\1_{r<r_1}}{\underline{\alpha} (r)} \quad\text{ if } r_0 < r < r_2,
\end{equation}
or
\begin{equation}
\label{eq:73''} g'(s) \varphi(s) \le -2c \sup_{r \in \hat{I}_s} \frac{a +\Phi (r)}{\underline{\alpha} (r)} - \sup_{r \in \hat{I}_s} \frac{M \1_{r < r_1}}{\underline{\alpha} (r)}\quad \text{ if } 0 < s \le r_2 \,.
\end{equation}
In \eqref{eq:5g}, \eqref{eq:6g} respectively, the function $g$ has been defined in such a way that this condition is satisfied. 
Now, by combining \eqref{eq:7*}, \eqref{eq:72} and \eqref{eq:73}, and bounding  the term $-a \pi(x,y)$ in \eqref{eq:7*} by zero, we obtain for $r \in (r_0, r_2)$:
\begin{equation}
\label{eq:73a} \E_{x,y} [ f(R') - f(r) ] \le - M \1_{r < r_1} - cf(r) \,.
\end{equation}

In the setup of Theorem \ref{thm:1}, we have chosen $M=0$ and $\rho(x,y) = f(d(x,y))$. Hence, \eqref{eq:73a} implies the assertion
\begin{equation}
\label{eq:7ba} \E_{x,y} [\rho(X',Y') ] \le (1-c)  \rho(x,y) \qquad\text{for } r = d(x,y) \in (r_0, r_2) \,.
\end{equation}

In the setup of Theorem \ref{thm:2}, by Assumption \ref{enum:a4}, 
\begin{equation}
\label{eq:7bb} \E_{x,y} [V(X') + V(Y')] \le (1-\lambda) (V(x) + V(y)) + 2C \,.
\end{equation}
Since $\rho(x,y) = f(d(x,y)) + \frac{M}{2 C} (V(x) + V(y)) \1_{x \neq y} $, we obtain for $r\in (r_0,r_1)$:
\begin{eqnarray}
\label{eq:73b} \lefteqn{\E_{x,y} [\rho(X',Y')]}\\  \nonumber 
&\le &-M + (1-c) f(d(x,y)) + \frac{M}{2C} (1-\lambda )(V(x) + V(y))+ M \\ & \le &(1-c) \rho(x,y).\nonumber
\end{eqnarray}
Here the last inequality holds since $\lambda \ge c $. 
On the other hand, for $r \in [r_1,r_2)$, we have $V(x) + V(y) \ge {4C}/{\lambda}$ by \eqref{eq:6r1}. Hence in this case, by \eqref{eq:7bb}, 
\begin{equation}
\label{eq:7bba} \E_{x,y}[V(X') + V(Y')]\ \le\ \left(1-{\lambda}/{2}\right) (V(x) + V(y)) \,.
\end{equation}
Since $c \le {\lambda}/{2}$, \eqref{eq:73a} and \eqref{eq:7bba} then again imply
\begin{eqnarray}
\label{eq:7bc} {\E_{x,y} [\rho (X',Y')]} &\le &(1-c) f(d(x,y)) + \left(1-\frac{\lambda}{2}\right) \frac{M}{2C} (V(x) + V(y)) \\ \nonumber & \le & (1-c) \rho (x,y) \,.
\end{eqnarray}

\emph{Case $r \le r_0$.} Noting that $f'' \le 0$ and $f' \le 1$ and applying \ref{enum:a2}, we see that for $r \in (0, r_0)$, \eqref{eq:7*} implies 
\begin{equation}
\label{eq:74a} \E_{x,y}[f(R') - f(r)] \le -a \pi (x,y) + \beta(x,y) \le - M - c(a+ r_0) \le -M - cf(r)
\end{equation}
provided $c \le \frac{1}{2} \pi (x,y)$ and 
\begin{equation}
\label{eq:74'} \frac{a}{2} \pi (x,y) \ge \frac{r_0}{2} \pi (x,y) + M + \beta(x,y) \,.
\end{equation}
This condition is satisfied by our choice of $a$, cf.\ \eqref{eq:6c}. Again, using \eqref{eq:7bb}, and since $c \le {\lambda}/{2}$, we obtain
\begin{multline*}
\E_{x,y} [\rho(X',Y') ] \\\le -M + (1-c) f(r) + (1-\lambda) \frac{M}{2C} (V(x) + V(y)) + 2C \frac{M}{2C} \le (1-c) \rho(x,y) \,.
\end{multline*}
\emph{Case $r \ge r_2$.} Here, we use the bound
$
f(R') - f(r) \le (R' - r) f'(r)
$
yielding 
\begin{equation}
\label{eq:7**} \E_{x,y} [f(R') - f(r)] \le \beta(x,y) f'(r).
\end{equation}

Now, we consider first the setup of Theorem \ref{thm:1}. Here, for $r \ge r_2$, we have
\[
\E_{x,y} [f(R') - f(r) ] \le \beta(r) f'(r) \le \beta(r) \frac{\varphi(r_1)}{2} \,,
\]
where we have used that by \eqref{eq:50}, $\beta(r) \le 0$, $f' \ge \nicefrac{\varphi}{2}$ and $\varphi$ is constant on $[r_1, \infty)$. To prove that the right hand side is bounded from above by $-cf(r)$, it is sufficient to show
\[
c(a+\Phi (r)) \le -\beta(r) \frac{\varphi(r_1)}{2} \tforall r \ge r_2 \,.
\]
We claim that this holds by the definition of $r_2$. Indeed, by the definition of $c$,
$$ c^{-1} \ge 4 \int_0^{r_2} \frac{1}{\varphi(s)} \sup_{u \in \hat{I}_s} \frac{a +\Phi (u)}{\underline{\alpha} (u)} \dd s \\\ge 4 \int_{r_1}^{r_2}\frac{\Phi (s)}{\varphi(s)\underline{\alpha}(s)} \dd s = \frac{4}{\varphi(r_1)} \int_{r_1}^{r_2} \frac{\Phi (s)}{\underline{\alpha} (s)} \dd s \,.
$$
Hence, by \eqref{eq:5d}, 
\[
c(a+\Phi (r)) \le \frac{1}{4} \left(\int_{r_1}^{r_2} \frac{\Phi(s) }{\underline{\alpha} (s)} \dd s \right)^{-1} (a+ \Phi (r)) \varphi(r_1) \le -\beta(r) \frac{\varphi(r_1)}{2},
\]
and thus
\[
\E_{x,y} [\rho (X', Y') - \rho (x,y) ] \le -c (a+\Phi (r)) \le -c \rho(x,y) \,.
\]

Finally, we now show contractivity for $r \ge r_2$ under the conditions in Theorem \ref{thm:2}. Here, by \eqref{eq:7bba} and \eqref{eq:7*},
\begin{eqnarray}
\E_{x,y} [\rho(X',Y')] &=& \E_{x,y} [f(R')] + \frac{M}{2C} \E_{x,y} [V(X') + V(Y')]\\
 \nonumber &\le & f(r) + \overline{\beta} (r) f'(r) + \frac{M}{2C} \left(1-\frac{\lambda}{2}\right) (V(x) + V(y)) \,.
\end{eqnarray}
Since $c \le {\lambda}/{4}$ by its definition, we obtain
\begin{equation}
\label{eq:75} \E_{x,y}[\rho(X',Y')] \le (1-c) \left(f(r) + \frac{M}{2C} (V(x) + V(y))\right) = (1-c) \rho(x,y)
\end{equation}
provided
\begin{equation}
\label{eq:75i} cf(r) + \overline{\beta} (r)f'(r) \le \frac{M}{2C} \frac{\lambda}{4} (V(x) + V(y)) \,.
\end{equation}
However, due to our choice of $r_2$ in (\ref{eq:6e}) and since $f' \leq \varphi$, we have
\[
\overline{\beta} (r) f'(r) \le \frac{\lambda M}{16C} (V(x) + V(y))\quad \text{ if } \lvert x-y \rvert \ge r_2 \,.
\]
Moreover, due to our choice of $c$ in (\ref{eq:6f}) and since $f \leq \Phi$, we get
\begin{equation*}
cf(r) \leq \frac{\lambda M}{16C} (V(x) + V(y))\quad \text{ if } \lvert x-y \rvert \ge r_2 \,.
\end{equation*}
Hence (\ref{eq:75i}) is indeed satisfied for $r \geq r_2$ and the proof is complete.

\section{Proof of Theorem \ref{thm:3}}\label{sec:proof3}
For proving Theorem \ref{thm:3}, we proceed in a similar way as in the proofs of Theorem \ref{thm:1} and Theorem \ref{thm:2} above. Suppose that conditions \ref{enum:b1}, \ref{enum:b2} and \ref{enum:b3} hold. Now, the intervals $I_r$, $r \in (0,\infty)$ are given by \eqref{eq:17} and we consider the dual intervals $\hat{I}_s$, $s \in (0,\infty)$ defined by
\begin{equation}
\label{eq:80} \hat{I}_s = \{ r \in (0,\infty) : s \in I_r\}.
\end{equation}
By \eqref{eq:18}, $I_r = (r-\ell (r), r)$ for $r \ge r_0$ and $I_r \subseteq (0,2r_0)$ for $r < r_0$. Therefore
\begin{align}
\label{eq:8.} \hat{I}_s &= \{ r > s : r-\ell(r) < s \}\qquad\text{ for } s \ge 2r_0,\ \text{ and} \\
\label{eq:8..} \hat{I}_s &\subseteq \{ r > s : r-\ell(r) < 2r_0 \}\qquad\text{ for } s< 2r_0 \,.
\end{align}
Let $\overline{\gamma} (r) = {2 \ob (r)}/{\ua (r)}$ as in \eqref{eq:5stara}. Similarly as in the proof of Theorem \ref{thm:1} and Theorem \ref{thm:2}, we assume that $\tilde{\gamma} \colon [0,\infty) \rightarrow [0,\infty)$ is a  function satisfying
\begin{align}
\label{eq:8**i} \sup_{\hat{I}_s} \overline{\gamma} &\le \tilde{\gamma} (s) \tforall s \in (0,\infty) \,, \\
\label{eq:8**ii} 4\sup\limits_{\hat{I}_s} \overline{\gamma} &\le \tilde{\gamma} (s) \tforall s \in (0,2r_0) \,, \quad\text{ and}\\
\label{eq:8**iii} \int_0^{2r_0} \tilde{\gamma} (s) \dd s &\le \log 2 \, .
\end{align}
Note the additional factor $4$ that has been introduced for technical reasons for $s < 2r_0$. In applications, this will usually not affect the bounds too much, as typically $r_0$ is a small constant. Condition \eqref{eq:8**iii} can always be satisfied by choosing $r_0$ small enough. As in \eqref{eq:choice r1}, we set
\begin{equation}
\label{eq:8a}
r_1 \coloneqq \sup \{ r > 0 : \tilde{\gamma} (r) > 0\},
\end{equation}
where $\sup \emptyset = 0$. Similarly as below \eqref{eq:choice r1}, by Assumption \ref{enum:b3}, we can choose $\tilde{\gamma}$ such that $r_1$ is finite.
 The metric is chosen similarly as in the proof of Theorem \ref{thm:1} above, where now $a=0$.
We define
\begin{equation}
\label{eq:8b} f(r) = \int_0^r \varphi(r \wedge r_2) g (s \wedge r_2) \dd s \,.
\end{equation}
Here
\begin{equation}
\label{eq:8c} \varphi(r) = \exp \left(-\int_0^r \tilde{\gamma} (s) \dd s \right),\qquad
 \Phi (r)= \int_0^r \varphi(s) \dd s,
\end{equation}
the constant $r_2$ is chosen such that 
\begin{equation}
\label{eq:8e} \frac{-\ob(r)}{\Phi (r)} \ge \frac{1}{8} \left(\int_{r_1}^{r_2} \frac{\Phi (s)}{\ua(s)} \dd s \right)^{-1} \text{ for } r \ge r_2,
\end{equation}
and
\begin{equation}
\label{eq:8g} g(r) = 1-2c \int_0^r \frac{1}{\varphi(s)} \sup_{u \in \hat{I}_s} \frac{\Phi (u)}{\ua (u)} \dd s,
\end{equation}
where the contraction rate $c$ is given by
\begin{equation}
\label{eq:8f} c = \frac{1}{4} \left(\int_0^{r_2} \frac{1}{\varphi(s)} \sup_{u \in \hat{I}_s} \frac{\Phi (u)}{\ua(u)} \dd s\right)^{-1} \,.
\end{equation}

\begin{proof}[Proof of Theorem \ref{thm:3}]
Let $x,y \in \R^d$ and $r = d(x,y)$.\smallskip\\ 
For $r \ge r_2$, \eqref{eq:21} follows in the same way as in the proof of Theorem \ref{thm:1} (with $a=0$). The crucial assumption for this is \eqref{eq:8e}, which holds due to \ref{enum:b3}, by analogy to \eqref{eq:5d} in the proof of Theorem \ref{thm:1}, which holds due to \ref{enum:a3}. \smallskip\\ 
Now assume that $r < r_2$. To prove \eqref{eq:21}, we show that
\begin{equation}
\label{eq:8**} \E_{x,y} [f(R')-f(r)] \le \ob (r) f'(r) + \frac{1}{2} \ua (r) \sup_{ I_r} f''\ \le\ -c f(r) \,.
\end{equation}
The first inequality follows similarly as in the proof of Theorem \ref{thm:1}, cf.\ \eqref{eq:7*}. To prove the second inequality, note that on $(0,r_2)$, 
\begin{align}
\label{eq:8***} f' &= g \varphi \,, & f'' &= g\varphi' + g' \varphi & \text{ and }& & f &\le \Phi \,. 
\end{align}
By \eqref{eq:8***} it is sufficient to show that $\varphi$, $g$ and $c$ have been chosen in such a way that
\begin{align}
\label{eq:8b1} \sup_{ I_r} (g \varphi') &\le -2 \frac{\ob (r)}{\ua (r)} g(r) \varphi(r)\\
\label{eq:8b2} \sup_{ I_r} (g' \varphi ) &\le -2 c\frac{\Phi  (r)}{\ua (r)} \,.
\end{align}
Then, by \eqref{eq:8**} we can conclude that
\[
\E_{x,y} [\rho(X',Y') - \rho(x,y) ] = \E_{x,y} [f(R') - f(r)] \le - c\Phi  (r) \le -c f(r) = -c \rho(x,y) \,.
\]
We first verify \eqref{eq:8b1}. This condition is satisfied provided
\begin{equation}
\label{eq:8b3} g(s) \varphi'(s) \le - \sup_{ \hat{I}_s} (\overline{\gamma} g \varphi ) \tforall s \le r_2.
\end{equation}
For $s \ge 2r_0$ we have
\[
\sup_{\hat{I}_s} (\overline{\gamma} g \varphi) \le \left(\sup_{\hat{I}_s} \overline{\gamma} \right)\left(\sup_{\hat{I}_s} g \varphi\right) \le \tilde{\gamma} (s) g(s) \varphi(s) \,,
\]
because $\hat{I}_s \subseteq (s,\infty)$ by \eqref{eq:8.} and since $g\varphi$ is decreasing. Hence, \eqref{eq:8b3} holds by definition of $\varphi$.

For $s < 2r_0$ we have to argue differently, since, in general, {$\hat{I}_s$ is not contained in $(s,\infty)$} in this case. Observe first that if $\sup_{\hat{I}_s} (\overline{\gamma} g \varphi) \leq 0$, then \eqref{eq:8b3} holds trivially since $\varphi$ is decreasing. Hence it is sufficient to consider the case of $\sup_{\hat{I}_s} (\overline{\gamma} g \varphi) > 0$. Noting that $g \varphi \le 1$, we have by \eqref{eq:8**ii}
\[
\sup_{\hat{I}_s} (\overline{\gamma} g \varphi) \le \sup_{\hat{I}_s} \overline{\gamma}^+\le \frac{1}{4} \tilde{\gamma} (s) 
\]
and hence, since $g \ge \frac{1}{2}$,
\[
g(s)\varphi'(s) \le \frac{1}{2} \varphi'(s) = - \frac{1}{2} \tilde{\gamma} (s) \varphi(s) \le -2 \varphi(s) \sup_{\hat{I}_s} (\overline{\gamma} g \varphi) \,.
\]
Thus, \eqref{eq:8b3} holds for $s < 2r_0$ since by \eqref{eq:8**iii},
\[
\varphi(s) = \exp \left(-\int_0^s \tilde{\gamma} (u) \dd u \right) \ge \frac{1}{2} \,.
\]
We thus have shown that \eqref{eq:8b3} and hence \eqref{eq:8b1} are satisfied. It remains to verify \eqref{eq:8b2}. This condition holds provided
\begin{equation}
\label{eq:8b4} g'(s) \varphi(s) \le \inf_{\hat{I}_s} \frac{-2c\Phi }{\ua}, \tforall s \le r_2
\end{equation}
or, equivalently,
\begin{equation}
\label{eq:8b5} g'(s) \le -2c \frac{1}{\varphi(s)} \sup_{\hat{I}_s} \frac{\Phi }{\ua}, \tforall s \le r_2 \,.
\end{equation}
The function $g$ has been chosen in \eqref{eq:8g} in such a way that this condition is satisfied.

\end{proof}

\section{Proof of perturbation result}\label{sec:proof4}
We now prove the perturbation result in Theorem \ref{thm:4}. Let $x,y \in S$, $x \neq y$. By \eqref{eq:26}, \eqref{eq:27}, \eqref{eq:28} and \eqref{eq:29},
\begin{eqnarray*}
\lefteqn{\E_{x,y} [\tilde{\rho}(\tilde{X}, \tilde{Y}) - \tilde{\rho}(x,y)] \ \le\ \E_{x,y} [\rho(\tilde{X}, \tilde{Y})-\rho(x,y)] - \frac{2b}{p} \P_{x,y} [\tilde X=\tilde Y]} \\
&\le &\E_{x,y}[(d(\tilde{X}, \tilde{Y}) - d(X', Y'))^+] 
 +  \E_{x,y} [\rho(X',Y') - \rho(x,y)] -\frac{2b}{p} \P_{x,y} [\tilde X=\tilde Y] \\
&\le &b -\frac{c}{2} \rho(x,y) - 2b \1_{d(x,y) < r_0} \,.
\end{eqnarray*}
{Note that in the second inequality we have used that $f$ is a contraction.} For $d(x,y) < r_0$ we obtain
\begin{equation}
\label{eq:31*} \E_{x,y} [\tilde{\rho}(\tilde{X}, \tilde{Y}) - \tilde{\rho}(x,y)] \le - \frac{p}{2} \frac{2b}{p} - \frac{c}{2} \rho(x,y) \le -\frac{1}{2} \min (c,p) \tilde{\rho}(x,y) \,.
\end{equation}
For $d(x,y) \ge r_0$, we use the fact that $b = \nicefrac{cf(r_0)}{4}$. Hence,
\[
\tilde{\rho} (x,y) = \rho(x,y) + \frac{2b}{p} \le \left(1+\frac{c}{2p}\right) \rho(x,y) \le \max \left(2, \frac{c}{p}\right) \rho(x,y),\qquad\text{and}
\]
\begin{equation}
\label{eq:32*} \E_{x,y} [\tilde{\rho} (\tilde{X}, \tilde{Y}) - \tilde{\rho}(x,y)] \le b-\frac{c}{2} \rho(x,y) \le -\frac{c}{4} \rho(x,y) \le -\frac{1}{8} \min (c,2p) \tilde{\rho} (x,y).
\end{equation}
The assertion of Theorem \ref{thm:4} follows from \eqref{eq:31*} and \eqref{eq:32*}.

\section{Proof of results for the Euler scheme}\label{sec:proofsEuler}
In this section, we prove the contraction results for the Euler scheme.
\begin{proof}[Proof of Lemma \ref{lem:6} \ref{enum:lem6i}, \ref{enum:lem6ii} and \ref{enum:lem6iii}]
We start with reduction steps. At first, we observe that the definitions of $\hat{\beta} (x,y), \hat{\alpha} (x,y)$ and $\hat{\pi} (x,y)$ only depend on $\hat{r} = \lvert \hat{x} - \hat{y}\rvert$ and $R' = \lvert X' -Y'\rvert$. Thus, the assertions \ref{enum:lem6i} \ref{enum:lem6ii}, \ref{enum:lem6iii} are statements about the coupled random walk transition step $(\hat{x}, \hat{y}) \rightarrow (X', Y')$ defined by \eqref{eq:0c}, and we may assume w.l.o.g.\ that $(\hat{x}, \hat{y}) = (x,y)$. Furthermore, $\hat{r}$ and the law of $R'$ under $\P_{x,y}$ are invariant under translations and rotations of the underlying state space $\R^d$. Therefore, we may even assume w.l.o.g.\ that $\hat{x} = x = 0$ and $\hat{y} = y = r e_1$, where $r =  \hat{r} $ and $e_1, \dots, e_d$ denotes the canonical basis of $\R^d$. Then
\begin{eqnarray}\label{eq:65a}
X' \ =\ \sqrt{h} Z, \qquad Y_\text{refl}'& =& re_1 + \sqrt{h} (I_d - 2e_1e_1^T)Z, \qquad\text{and}\\
\label{eq:200} \nicefrac{\phi_{\hat{y},hI} (X')}{\phi_{\hat{x}, hI} (X')} &=& \nicefrac{\phi_{r,h} (X_1')}{\phi_{0, h} (X_1')} \,,
\end{eqnarray}
{where $X_i' = e_i^TX'$.} Thus, by \eqref{eq:0c}, $Y_i' = X_i'$ for $i \ge 2$, and
\begin{align}
\label{eq:201} Y_1' &= \begin{cases} X_1' & \text{ if } U\le {\phi_{r,h} (X_1') }/{\phi_{0,h} (X_1')}, \\ r -\sqrt{h} Z &\text{ otherwise.}\end{cases}
\end{align}
In particular, 
$ R' = \lvert X'-Y'\rvert = \lvert (X'-Y') e_1\rvert = \lvert X_1' - Y_1'\rvert $.
Since this is distributed as in the one-dimensional case, we may assume w.l.o.g.\ $d=1$.\smallskip

We are now left with a simple one-dimensional problem where $x = 0$, $y = r$, and $\hat{r} = r = \lvert x-y \rvert$. The coupling is given by 
\begin{eqnarray}
\label{eq:201b} X'&=&\sqrt hZ,\qquad Y' \ =\ \begin{cases} X' & \text{ if } U\le {\phi_{r,h} (X') }/{\phi_{0,h} (X')}, \\ r -X' &\text{ otherwise,}\end{cases}
\end{eqnarray}
where $Z \sim N(0,1)$ and $U \sim \Unif (0,1)$ are independent. 
Hence $X'\sim N(0,h)$, the conditional probability given $Z$ that $Y' = X'$ is $\min (1,{\phi_{r,h}}(X')/{\phi_{0,h} (X')})$, and if $Y' \neq X'$, then $R' = \lvert X' -Y' \rvert = \lvert r - 2X' \rvert$. Since $ \phi_{r,h}(t)\le\phi_{0,h}(t)$ if and only if $t\le \nicefrac{r}{2}$, we obtain
\begin{eqnarray*}
\E_{x,y} [R'] &=& \int_{-\infty}^{\infty } |r-2t |\,  (1-{\phi_{r,h}(t)}/{\phi_{0,h} (t)})^+\,  \phi_{0,h} (t) \, d t \\
&=& \int_{-\infty}^{{r}/{2}} (r-2t )  ({\phi_{0,h}(t)}-{\phi_{r,h} (t)})  \, d t \\
&=& \frac 12\int_{-\infty}^{\infty} (r-2t )  ({\phi_{0,h}(t)}-{\phi_{r,h} (t)})  \, d t 
\ =\ r.
\end{eqnarray*}
Here we have used in the third step that the integrand is symmetric w.r.t.\ $t=r/2$, i.e., invariant under the 
transformation $t\mapsto r-t$.
Thus $\hat\beta (x,y)=\E_{x,y} [R'-r]=0$, which proves Assertion (i).\smallskip

Next,
we are going to prove the lower bound for $\hat{\alpha} (x,y)$. Recall from \eqref{eqchoicer0}
and \eqref{eq:int} that $I_r = (0, r+ \sqrt{h})$ for $r < \sqrt{h}$ and $I_r = (r -\sqrt{h}, r)$ for $r \ge \sqrt{h}$. We first consider the case $r \ge \sqrt{h}$. 
Similarly as above, we obtain 
\begin{eqnarray}
\nonumber
\hat{\alpha} (x,y) &\ge & \E_{0,r}[(R'-r)^2 ; R' \in I_r] \\
\nonumber &\ge &\int_{-\infty}^{{r}/{2}} (r-2t -r)^2 \1_{ I_r}(r-2t)\,  ({\phi_{0,h}(t)}-{\phi_{r,h} (t)})\,  d t \\
\label{eq:204}&= &4 \int_{0}^{{\sqrt{h}}/{2}} t^2\,  (1-e^{\nicefrac{(rt - {r^2}/{2})}{h}})\, \phi_{0,h} (t) \, d t \\
\nonumber &=& 4h \int_0^{{1}/{2}} u^2\, (1-e^{\frac{r}{\sqrt{h}} (u-\frac{r}{2\sqrt{h}})})\, \phi_{0,1}(u) \,d u \\
\nonumber &\ge & 4h \int_0^{{1}/{2}}u^2\, (1-e^{u-1/2})\, \phi_{0,1} (u) \,d u.
\end{eqnarray}
Here we have used in the last step that $s \mapsto s (u-{s}/{2})$ is decreasing for $s\ge u$, and $r / \sqrt{h} \ge 1\ge u$ for $u\in [0,1/2]$.
Note that in the second step we only use the reflection behaviour of the coupling. This is due to the fact that the contribution from jumping to the same point would be of negligible order in $h$. 
Now assume $r < \sqrt{h}$. Then $r-2t \in I_r$ if and only if $t \in (-\frac{\sqrt{h}}{2},\frac{r}{2})$. Thus, 
\begin{eqnarray}
\nonumber
\hat{\alpha} (x,y) &\ge &\int_{-{\sqrt{h}}/{2}}^0 (r-2t-r)^2\, ({\phi_{0,h}(t)} -\phi_{r,h} (t))  \,d t \\
\label{eq:205}&= &4h \int_{-{1}/{2}}^0 u^2\, (1-e^{\frac{r}{\sqrt{h}} (u-\frac{r}{2\sqrt{h}})})\, \phi_{0,1}(u) \,d u \\
\nonumber &\ge &4(1-e^{-1} )\,  h \frac{r}{\sqrt{h}} \int_0^{{1}/{2}} u^3\, \phi_{0,1} (u) \,d u.
\end{eqnarray}
Here, we have used in the last step that for $r < \sqrt{h}$ and $u \in [-\nicefrac{1}{2},0]$, we have $s \coloneqq \frac{r}{\sqrt{h}} (u - \frac{r}{2\sqrt{h}}) \in [-1,0]$ and hence
$
e^s - 1\le (1-e^{-1}) s
$.
By combining \eqref{eq:204} and \eqref{eq:205}, we obtain
$
\hat{\alpha} (x,y) \ge c_0 \min (r,\sqrt{h}) \sqrt{h}
$,
where
\[
c_0 = 4\min \left(\int_0^{{1}/{2}} u^2 (1-e^{u-1/2}) \phi_{0,1} (u) \,d u ,\, (1-e^{-1}) \int_0^{{1}/{2}} u^3 \phi_{0,1} (u)\,d u \right) \ge 0.007.
\]
This proves Assertion (ii).\smallskip 

Finally, for $X' \ge {r}/{2}$, we have $\phi_{r,h} (X') \ge \phi_{0,h} (X')$,
and hence $Y'=X'$. Thus,
\begin{eqnarray*}
\pi (x,y) &=& \P_{x,y} [R' = 0] \ =\  \P_{x,y} [X'  = Y']\ \ge\ \P_{x,y} [X' \ge {r}/{2}] \\
&=&\int_{{r}/{2}}^\infty \phi_{0,h}(t) \,d t\ =\ \int_{\frac{r}{2\sqrt{h}}}^\infty \phi_{0,1} (t) \, d t\  \ge\ \int_1^\infty \phi_{0,1}(t) \,d t \
\ge\ 0.15
\end{eqnarray*}
provided $\hat{r} = r \le 2 \sqrt{h}$. Therefore, Assertion (iii) holds as well.
\end{proof}

\begin{proof}[Proof of Lemma \ref{lem:6} \ref{enum:lem6iv}, \ref{enum:lem6v},
\ref{enum:lem6vi} and \ref{enum:lem6vii}] Note that unlike in the proof of assertions \ref{enum:lem6i}-\ref{enum:lem6iii}, here it is important to consider $\hat{r} \neq r$. Assertions \ref{enum:lem6iv} and \ref{enum:lem6v} are straightforward consequences of Assertion \ref{enum:lem6i}. Indeed, by \eqref{eq:beta} and Lemma \ref{lem:6}\ref{enum:lem6i}, 
\begin{equation}
\label{eq:406} \beta(x,y)\ =\ \hat{\beta} (x,y) + \hat{r}-r \ =\  \hat{r}-r.
\end{equation}
Assuming \ref{enum:c1} and \ref{enum:c3}, this implies \ref{enum:lem6iv}, because
\begin{eqnarray}
\label{eq:407a} |\hat{r}-r|& \le& \lvert (\hat{x} - \hat{y})-(x-y) \rvert \ =\ \lvert  h(b(x) - b(y))\rvert\ \le\ hL r, \quad\text{and}
\\
\label{eq:407} \hat{r} &=& \sqrt{\lvert x-y \rvert^2 + 2h(x-y)\cdot (b(x) - b(y)) + h^2 \lvert b(x) - b(y) \rvert^2}\\
&\le &r \sqrt{1+2hJ + h^2 L^2 }\ \le\ r(1+hJ+\nicefrac{h^2L^2}{2}),\nonumber
\end{eqnarray}
where we use $\sqrt{1+x} \leq 1 + x/2$ for $x \ge -1$. Similarly, assuming \ref{enum:c2} and \ref{enum:c3}, \eqref{eq:406} implies \ref{enum:lem6v}, since $K\le L$ and thus $-2hK + h^2L^2 \geq - 1$ and
\begin{equation}
\label{eq:e1}\hat{r}\ \le\ r \sqrt{1-2hK + h^2L^2}\ \le\ r \left(1-hK + \nicefrac{h^2L^2}{2}\right)\text{\quad for }r \ge \mathcal{R}.
\end{equation}
In order to prove \ref{enum:lem6vi} we assume $\sqrt h\le r\le 1/(4L\sqrt h)$. Then by
\eqref{eq:407a}, $ |\hat{r}-r|\le \sqrt h/4$. Therefore, by a similar computation as in \eqref{eq:204},
\begin{eqnarray*}
\alpha (x,y) &=& \E_{0,r}[(R'-r)^2 ; R' \in (r-\sqrt h,r)]\
\ge \ \frac{h}{16}\, \P_{0,r}[ R' \in (r-\sqrt h,r-\frac{\sqrt h}4)]\\
&\ge & \frac{h}{16}\, \P_{0,r}[ R' \in (\hat r-\frac{3\sqrt h}4,\hat r-\frac{\sqrt h}2)]\
\ge \ \frac{h}{16}\, \int_{{\sqrt{h}}/{4}}^{{3\sqrt{h}}/{8}}  ({\phi_{0,h}(t)} -\phi_{\hat{r},h} (t))  \,d t \\
&= &\frac h{16} \int_{{1}/{4}}^{3/8}  (1-e^{\frac{\hat{r}}{\sqrt{h}} (u-\frac{\hat{r}}{2\sqrt{h}})})\, \phi_{0,1}(u) \,d u \
\ge \ \frac h{16}\int_{{1}/{4}}^{3/8}  (1-e^{ u-\frac{1}{2}})\, \phi_{0,1}(u) \,d u.
\end{eqnarray*}
This shows that \ref{enum:lem6vi} holds with $\tilde c_0:=\frac{1}{16}\int_{1/4}^{3/8}(1-e^{u-1/2})\phi_{0,1}(u)\, du\ge 0.0005$.\smallskip\\
Finally, Assertion \ref{enum:lem6vii} is a direct consequence of Assertion \ref{enum:lem6iii}, since by 
\eqref{eq:407a}, $|\hat r-r|\le \sqrt h$ if $r\le\sqrt h$ and $h\le 1/L$.
\end{proof}

The following proof of Theorem \ref{thm:7} follows the argumentation in the proofs of Theorems \ref{thm:1} and \ref{thm:3} in the case $r\le\mathcal R$. For $r>\mathcal R$, the contractivity is shown by a direct argument based on Lemma \ref{lem:6} (v).

\begin{proof}[Proof of Theorem \ref{thm:7}]
Let $x,y \in \R^d$, $h\in (0,h_0]$ and $a\in \{ 0\}\cup [\sqrt h,\infty )$. \smallskip

\emph{(i)}. We first consider the case where $r = \lvert x-y\rvert > \mathcal R$. By the choice of $h_0$ in the statement of the theorem, $h \le \nicefrac{K}{L^2}$. Therefore, by Lemma \ref{lem:6}, 
\[
\E_{x,y} [R'-r]\ =\ \beta(x,y) \ \le\ -\left(K-\nicefrac{L^2h}{2}\right)hr\ \le\ -\nicefrac{Khr}{2}.
\]
Since $f_a$ is concave with $f'_a \ge \nicefrac{1}{2}$, we immediately obtain
\begin{equation}\label{eq:a130}
\E_{x,y} [f_a(R') -f_a(r) ]\ \le\ f'_a(r) \E_{x,y} [R'-r]\ \le\ -\nicefrac{Khr}{4} ,
\end{equation}
and hence, as $f_a(r)\le r$ and $r>\mathcal R$,
\begin{equation}\label{eq:a130b}
\E_{x,y} [ \rho_a(X',Y') - \rho_a(x,y) ] \ \le\   \frac{-Khr/4}{a+f_a(r)}\rho_a(x,y)\ \le\ \frac{-Kh/4}{1+a/\mathcal R}\rho_a(x,y).
\end{equation}

\emph{(ii)}. Now suppose $r \le \mathcal R$. Since $\hat{r} \le r$ by \eqref{eq:30*}, we have
\begin{equation}
\label{eq:a13} \E_{x,y}[\rho_a(X',Y') - \rho_a(x,y) ]\ \le\ \E_{x,y} [\rho_a(X',Y') - \rho_a(\hat x,\hat y)] \,.
\end{equation}
We can now apply the arguments in the proofs of Theorems \ref{thm:1} and \ref{thm:3} with $\alpha$ and $\beta$ replaced by the corresponding quantities $\hat{\alpha}$ and $\hat{\beta}$ for the coupled random walk transition $(\hat{x}, \hat{y}) \mapsto (X',Y')$, with $r_2$ and $r_1$ replaced by $\mathcal R$. Indeed,
note that since the case of $r>\mathcal R$ has already been considered above, we only need to use the parts of the proofs of Theorems \ref{thm:1} and \ref{thm:3} concerned with the case of $r \leq \mathcal{R}$ and thus Assumptions \ref{enum:a3} and \ref{enum:b3} are not required.\smallskip

We consider first $a=0$. In this case, we can proceed as in the proof of Theorem \ref{thm:3} with $r_0=\sqrt h$. By Lemma \ref{lem:6}, we can choose
$\ua (\hat{r}) = c_0 \min (\hat{r}, \sqrt{h}) \sqrt{h}$,
 $\ob \equiv 0$, $ \overline{\gamma} \equiv 0$, $  \tilde{\gamma} \equiv 0$,
 $ \varphi \equiv 1$,
 \begin{equation}\label{eq:7**f}
 \Phi (u)=u,\quad g_0(u)  = 1-2c \int_0^u \sup_{\hat{I}_s} \frac{\Phi}{\ua}\, ds,\quad\text{and}\quad
 c =  \frac{1}{4} \left(\int_0^{\mathcal{R}} \sup_{\hat{I}_s} \frac{\Phi}{\ua} \dd s \right)^{-1} 
 \end{equation}
in order to satisfy \eqref{eq:8**i}, \eqref{eq:8**ii}, \eqref{eq:8**iii}, \eqref{eq:8c}, \eqref{eq:8g} and \eqref{eq:8f}. Here $\hat I_s$ is defined by \eqref{eq:80}. With these choices we obtain as in the proof of Theorem \ref{thm:3}
\begin{equation}
\label{eq:a14} \E_{x,y}[f_0(R')-f_0(\hat{r}) ] \le -c f_0(\hat{r}) \qquad\text{for } \hat{r} \le \mathcal R,
\end{equation}
where $f_0$ is defined by \eqref{eq:31}. Noting that $\hat{r} \le r$ by \eqref{eq:30*}, 
the bounds in \eqref{eq:a13} and \eqref{eq:a14} now imply that for $r \le \mathcal R$,
\begin{equation}
\label{eq:a15} \E_{x,y} [\rho_0(X',Y')]= 
\E_{x,y}[f_0(R')]  \le (1-c) f_0(\hat{r}) \le (1-c) f_0(r)  = (1-c)\rho_0(x,y).
\end{equation}
It only remains to show $c\ge c_1(0)h$.
Suppose first that $s <  2\sqrt{h}=2r_0 $. Then $\hat{I}_s \subseteq (0,3\sqrt{h})$. Since  $\Phi(u) =u$ and $\ua (u) = c_0 \min (u,\sqrt{h}) \sqrt{h} \ge c_0 u \sqrt{h}/3$ for $u < 3 \sqrt{h}$, we obtain
\begin{equation}
\label{eq:a16} \sup_{\hat{I}_s} \frac{\Phi}{\ua}\ \le\ \sup_{u < 3 \sqrt{h}} \frac{u}{\ua (u)} \ \le\ 3 c_0^{-1} h^{-\nicefrac{1}{2}}\tforall s < 2 \sqrt{h}.
\end{equation}
For $s \ge 2\sqrt{h}$, $\hat{I}_s = (s,s+\sqrt{h})$. Hence $\ua \equiv c_0 h$ on $\hat{I}_s$, and 
\begin{equation}
\label{eq:a17}\sup_{\hat{I}_s} \frac{\Phi}{\ua} = c_0^{-1} h^{-1} (s+\sqrt{h}) \tforall s \ge 2 \sqrt{h}.
\end{equation}
By \eqref{eq:7**f}, \eqref{eq:a16}, \eqref{eq:a17} we see that
\begin{equation}
\label{eq:a18} c^{-1} 
\le 24 c_0^{-1} + 2c_0^{-1} h^{-1} \mathcal{R}^2 + 4c_0^{-1}h^{-\nicefrac{1}{2}} \mathcal{R} = 2c_0^{-1} h^{-1} (\mathcal{R}^2 + 2h^{\nicefrac{1}{2}} \mathcal{R} + 12 h).
\end{equation}
The assertion for $a=0$ now follows by \eqref{eq:a130}, \eqref{eq:a15} and \eqref{eq:a18}.\medskip

Now consider the case $a\ge \sqrt h$. Here we can proceed as in the proof of Theorem \ref{thm:1} with $r_0=\epsilon =\sqrt h$. We now choose the intervals $I_r$ and the dual intervals $\hat I_s$ according to \eqref{eq:501} and \eqref{eq:502}, i.e., $I_r=((r-\sqrt h)^+,r)$ and $\hat I_s=(\max (s,\sqrt h),s+\sqrt h)$. By Lemma \ref{lem:6}, we can choose
$\ua $,
 $\ob $, $ \overline{\gamma} $, $  \tilde{\gamma} $,
 $ \varphi$ and $\Phi $ as above so that conditions \eqref{eq:5stara}, \eqref{eq:5star}, \eqref{eq:503}, \eqref{eq:50}, \eqref{eq:5a}, \eqref{eq:5c}, \eqref{eq:5e}, \eqref{eq:5f} and \eqref{eq:5g} are satisfied. In particular, choosing $a\ge\sqrt h =r_0$ guarantees that 
\eqref{eq:5c} is satisfied since $\beta(x,y) \leq 0$ for all $x$, $y \in \mathbb{R}^d$.
  Note that for $u\in \hat I_s$ we have $\ua (u)\ge c_0h$, because $u\ge \sqrt h$. Setting
 \begin{equation}\label{eq:7**f2}
 g_a(u)  = 1-2c \int_0^u \sup_{\hat{I}_s} \frac{a+\Phi}{\ua}\, ds,\quad
 c =  \min\left( \frac{p_0}{2}\, ,\,\frac{1}{4} \left(\int_0^{\mathcal{R}} \sup_{\hat{I}_s} \frac{a+\Phi}{\ua} \dd s \right)^{-1} \right),
 \end{equation}
we obtain 
\begin{equation}
\label{eq:a15a} \E_{x,y} [\rho_a(X',Y')]\ \le\ (1-c) \rho_a(\hat x,\hat y)\ \le\ (1-c)  \rho_a(x,y),
\end{equation}
where $\rho_a$ is defined by \eqref{eq:31}. The bound $c \geq c_1(a)$ follows as in (\ref{eq:a17}) and (\ref{eq:a18}).
\end{proof}

\begin{proof}[Proof of Theorem \ref{thm:8a}]
	In order to apply Theorem \ref{thm:1}, we set $\varepsilon = r_0 = \sqrt{h}$, and hence $I_r = ((r-\sqrt{h})^{+},r)$ and $\hat{I}_s = (s \vee \sqrt{h}, s + \sqrt{h})$ for all $r$, $s > 0$. 
By Lemma \ref{lem:6}, condition \eqref{eq:11} is satisfied for $h\le h_0$ with
	\begin{equation}\label{eq:alphapib}
	\ua (r) = \tilde c_0 h\, \1_{\sqrt h\le r\le 1/(4L\sqrt h)}, \quad \up (r)=p_0\,1_{r\le\sqrt h} , \quad \text{and}
	\end{equation}
	\begin{equation}\label{eq:betab}
	 \ob (r) = \begin{cases} \Lambda h r &\text{for } r < \mathcal{R}, \\ -{K}hr/2, &\text{for } r \geq \mathcal{R}. \end{cases}
	 \end{equation}
Here we have used that by the assumptions, $h_0L\le 1$ and $h_0L^2\le K$. 
Moreover, the assumption on $h_0$ implies $1/(4L\sqrt h_0)\ge\mathcal R$ since $r_2 > \mathcal{R}$. Hence by \eqref{eq:5stara},	 
	$$ 
	  \overline{\gamma}(r) = 2\ob (r)/\ua (r) ={2}{\tilde c_0^{-1}}\Lambda r\qquad
	  \text{for }\sqrt h\le  r <\mathcal{R}, $$
	  $\overline{\gamma}(r)\le 0$ for $r\ge\mathcal R$, and thus
\eqref{eq:5star}, \eqref{eq:choice r1} and \eqref{eq:5a} are satisfied with	  
\begin{eqnarray}\nonumber
	\tilde{\gamma}(r)& =& {2}{\tilde c_0^{-1}} \Lambda (r + \sqrt{h})\, \1_{r<\mathcal R},\quad r_1=\mathcal R,\quad\text{ and}\\
 \label{eq:EulerGeneralThm1DefVarphi}	
	 \varphi(r) &=& \exp\left({ -{\tilde c_0^{-1}}{\Lambda} \left( (r \wedge \mathcal{R})^2 + 2 \sqrt{h} (r \wedge \mathcal{R}) \right)}\right)  \,.
\end{eqnarray}
	For $a\ge 2\sqrt h$, condition \eqref{eq:5c} is satisfied by \eqref{eq:alphapib},
	\eqref{eq:betab}, and since by assumption, $ \sqrt{h} + 2 \Lambda h^{3/2} / p_0\le 2\sqrt h\le a$ for $h\le h_0$.
	In order to verify (\ref{eq:5d}) we need to choose $r_2 \geq r_1 = \mathcal{R}$ such that 
	\begin{equation}\label{eq:proofEulerLyapunov1}
	2 \int_{\mathcal{R}}^{r_2} \frac{\Phi(s)}{\ua(s)} ds \geq \frac{a + \Phi(r)}{- \ob(r)} \qquad\text{for all } r \geq r_2 \,.
	\end{equation}
	To this end, note that for $r \geq \mathcal{R}$, we have $ \Phi(r) = \Phi(\mathcal{R}) + (r - \mathcal{R})\varphi(\mathcal{R})$. Furthermore, since $1/(4L\sqrt{h}) \leq r_2$ by assumption, on $[\mathcal{R}, r_2]$ we can use the formula for $\ua$ given in (\ref{eq:alphapib}). Hence (\ref{eq:proofEulerLyapunov1}) is satisfied if
	\begin{equation}\label{eq:proofEulerLyapunov2}
	\frac{2\Phi(\mathcal{R})}{\tilde c_0} (r_2 - \mathcal{R}) + \frac{\varphi(\mathcal{R})(r_2 - \mathcal{R})^2}{\tilde c_0} \geq \frac{a + \Phi (\mathcal R) +r\varphi (\mathcal R) }{Kr/2} \quad\text{ for } r \geq r_2 .
	\end{equation}
	Since we assume that $a \leq \Phi (\mathcal R)$, this condition holds if we choose
	\begin{equation}\label{eq:EulerGeneralThm1Defr2}
	r_2 \ =\  \mathcal{R}+\sqrt{2\tilde c_0/K} \,.
	\end{equation}
Hence from Theorem \ref{thm:1} we obtain $\mathbb{E}_{x,y}[\rho_a(X',Y')] \leq (1-c)\rho_a(x,y)$ with $c$ given by (\ref{eq:5f}), for $\rho_a(x,y) = \1_{x \neq y} + f_a(|x-y|)$, where 
	\begin{equation}\label{eq:EulerGeneralThm1Deffa}
	f_a(r) = \int_0^r \varphi(s \wedge r_2) g_a(s \wedge r_2) ds
	\end{equation}
	with $\varphi$ given by (\ref{eq:EulerGeneralThm1DefVarphi}) and $g_a = g$ given by (\ref{eq:5g}).
	Moreover, we can easily bound the second quantity appearing in the definition (\ref{eq:5f}) of $c$. Indeed, for $s<r_2$ and $u\in\hat I_s$ we have $\sqrt h<u<r_2+\sqrt h\le 1/(4L\sqrt h)$ for $h<h_0$. Therefore, $\ua (u)\ge \tilde c_0h$ by \eqref{eq:alphapib}.
	Since $\varphi (s)\ge \varphi (\mathcal R )$ and $\Phi (u)\le u$, we obtain
	\begin{equation*}
	\int_0^{r_2} \frac{1}{\varphi(s)} \sup_{u \in \hat{I}_s} \frac{a + \Phi(u)}{\ua (u)} ds \ 
	\leq\ \frac{1}{\tilde c_0 h \varphi(\mathcal{R})} \left( a r_2 + \int_0^{r_2} (s + \sqrt{h}) ds \right),
	\end{equation*}
	and hence
	\begin{equation*}
	c \geq \min \left( \frac{1}{2}p_0 ,  \frac{\tilde c_0 h \varphi(\mathcal{R})}{4r_2(a + \sqrt{h}) + 2r_2^2} \right) \,.
	\end{equation*}
	This implies the assertion, since by \eqref{eq:EulerGeneralThm1Defr2}, 
	$$r_2^2+2r_2(a+\sqrt h)\le 2\max (\mathcal R^2+2(a+\sqrt h)\mathcal R\, ,\, 2\tilde c_0K^{-1}+2(a+\sqrt h )\sqrt{2\tilde c_0/K}).$$
\end{proof}

In the following variation of Theorem \ref{thm:8a}, Condition \ref{enum:c2} is replaced by a Lyapunov condition:

\begin{thm}[Euler scheme, general case with Lyapunov condition]\label{thm:8b}
\	\\Suppose that Conditions \ref{enum:c1} and \ref{enum:c3} are satisfied and that the transition kernel $p$ of the Euler scheme satisfies Assumption \ref{enum:a4i} with a Lyapunov function $V$, i.e., there exist constants $C$, $\lambda > 0$ such that $pV \leq (1- \lambda)V + C$. Moreover, assume that $\lim_{r \to \infty} \frac{V(x) + V(y)}{r} = \infty$. Let $h_0 = \min \left( \left( \frac{2L}{p_0} + \frac{\tilde c_0 \varphi(r_1)}{4(r_1 + 1)} \right)^{-2}, (16L^2r_2^2)^{-1} \right)$, where $r_1$, $r_2 > 0$ are constants specified in (\ref{eq:proofEulerLyapunovDefr1}) and (\ref{eq:proofEulerLyapunovDefr2}). Suppose further that $a \in (2\sqrt{h}, r_2)$ and let $\rho_a(x,y) = (a + \frac{M}{2C}(V(x) + V(y))) \1_{x \neq y} + f_a(|x-y|)$ with $M$ given by (\ref{eq:EulerLyapunovDefM}) and $f_a$ defined in (\ref{eq:proofEulerLyapunovDeffa}). Let
	\begin{equation*}
	c_2(a) = \frac{1}{4} \min \left( \frac{2p_0}{h} , \frac{\lambda}{h} , 4\varphi(r_1) \Lambda, \frac{\tilde c_0 \varphi(r_2)}{2r_2(a + \sqrt{h}) + r_2^2} \right)
	\end{equation*}	
	with $\varphi$ given by (\ref{eq:proofEulerLyapunovDefVarphi}). 
	Then for all $h \in (0, h_0)$ we have 
	\begin{equation*}
	\E_{x,y}[\rho_a (X',Y')]\le\left(1-c_2(a)h\right) \rho_a (x,y)\quad
	\mbox{for all }x,y\in\r^d.
	\end{equation*}
\end{thm}

\begin{example}
	It is easy to see that if the drift $b$ satisfies a linear growth condition $|b(x)|^2 \leq L_0(1 + |x|^2)$ for all $x \in \mathbb{R}^d$ with a constant $L_0 > 0$ (which is implied by \ref{enum:c3} with $L_0 = 2 \max (L^2, |b(0)|^2)$) and a dissipativity condition 
	\begin{equation}\label{eq:dissipativity}
	\langle b(x) , x \rangle \leq M_1 - M_2 |x|^2 \text{ for all } x \in \mathbb{R}^d 
	\end{equation}
	with constants $M_1$, $M_2 > 0$, then the transition kernel $p$ of the Euler scheme satisfies the Lyapunov condition $pV \leq (1-\lambda)V + C$ with the Lyapunov function $V(x) = |x|^2$ and constants $\lambda = 2hM_2 - h^2 L_0$ and $C = h^2 L_0 + 2hM_1 + hd$, whenever $h < 2M_2/L_0$. Since the quadratic function satisfies the growth condition required in Theorem \ref{thm:8b} and the dissipativity condition (\ref{eq:dissipativity}) is significantly weaker than Assumption \ref{enum:c2}, we can apply this result to more general cases than the ones covered by Theorems \ref{thm:8a} and \ref{thm:8}.
\end{example}

\begin{proof}[Proof of Theorem \ref{thm:8b}]	
	Here we want to apply Theorem \ref{thm:2} and hence we need to verify the conditions listed in Subsection \ref{subsec:choiceOfMetricThm2}. Exactly as in the proof of Theorem \ref{thm:8a}, we choose $\varepsilon = r_0 = \sqrt{h}$ and we have the intervals $I_r = ((r-\sqrt{h})^{+},r)$ and $\hat{I}_s = (s \vee \sqrt{h}, s + \sqrt{h})$ for all $r$, $s > 0$. By Lemma \ref{lem:3} we get
	\begin{equation*}
	\ua (r) = \tilde c_0 h, \qquad \ob (r) = \Lambda h r, \qquad \text{and} \qquad \overline{\gamma}(r) = \frac{2}{\tilde c_0}\Lambda r .
	\end{equation*}
	Similarly as in the previous proof, the formula for $\ua(r)$ is valid for all $r \in (\sqrt{h}, r_2)$ since $h_0 \leq (16L^2r_2^2)^{-1}$, although here $r_2$ is given by (\ref{eq:proofEulerLyapunovDefr2}). Moreover, we have
	\begin{equation}\label{eq:proofEulerLyapunovDefVarphi}
	\tilde{\gamma}(s) = \frac{2}{\tilde c_0} \Lambda (s + \sqrt{h}), \qquad \text{and} \qquad \varphi(r) = \exp\left( - \frac{\Lambda}{\tilde c_0}\left( r^2 + 2 \sqrt{h} r \right) \right) \,.
	\end{equation}
	Now we choose $r_1$ as in (\ref{eq:6r1}), based on Assumption \ref{enum:a4i}. Namely,
	\begin{equation}\label{eq:proofEulerLyapunovDefr1}
	r_1 := \sup \left\{ |x-y| = r : x, y \in \mathbb{R}^d \,, V(x) + V(y) < 4C/\lambda \right\} \,.
	\end{equation}
	 In order for (\ref{eq:6b}) to be satisfied, it is sufficient to choose $M$ such that
	\begin{equation*}
	M \leq \frac{1}{4} \left( \int_0^{r_1} \frac{1}{\varphi(s)} \sup_{u \in \hat{I}_s} \frac{1}{\tilde c_0 h} ds \right)^{-1} = \frac{h \tilde c_0}{4} \left( \int_0^{r_1} \frac{1}{\varphi(s)} ds \right)^{-1} \,.
	\end{equation*}
	Note, however, that $\varphi(s) \geq \varphi(r_1)$ for all $s > 0$ and hence
	\begin{equation*}
	\frac{h \tilde c_0}{4} \left( \int_0^{r_1} \frac{1}{\varphi(s)} ds \right)^{-1} \geq \frac{h \tilde c_0}{4} \left( \int_0^{r_1} \frac{1}{\varphi(r_1)} ds \right)^{-1} = \frac{h \tilde c_0}{4} \frac{\varphi(r_1)}{r_1} \geq \frac{h \tilde c_0 \varphi(r_1)}{4(r_1 + 1)} \,.
	\end{equation*}
	Thus if we choose
	\begin{equation}\label{eq:EulerLyapunovDefM}
	M = \frac{h \tilde c_0 \varphi(r_1)}{4(r_1 + 1)} \,,
	\end{equation}
	then condition (\ref{eq:6b}) is indeed satisfied. Note that $r_1 + 1$ is chosen here instead of $r_1$ in order to prevent the value of $M$ from being too large when $r_1$ is very small (or even zero). Now condition (\ref{eq:6c}) reads as
	\begin{equation}\label{eq:proofEulerLyapunov3}
	a \geq \sqrt{h} + \frac{2}{p_0} \left( \Lambda h^{3/2} + \frac{h \tilde c_0 \varphi(r_1)}{4(r_1 + 1)} \right) \,.
	\end{equation}
	However, since we choose $h \leq h_0 \leq \left( \frac{2L}{p_0} + \frac{\tilde c_0 \varphi(r_1)}{4(r_1 + 1)} \right)^{-2}$, we see that (\ref{eq:proofEulerLyapunov3}) holds for all $a \geq 2 \sqrt{h}$. It remains to verify condition (\ref{eq:6d}), for which we need 
	\begin{equation*}
	V(x) + V(y) \geq \frac{16C}{\lambda M} \varphi(r) \Lambda h r 
	\end{equation*}
	to hold for all $r \geq r_2$. Since $\varphi$ is decreasing, using the choice of $M$ in (\ref{eq:EulerLyapunovDefM}), we see that it is sufficient to have 
	\begin{equation}\label{eq:proofEulerLyapunov4}
	V(x) + V(y) \geq \frac{16C}{\lambda} \frac{4(r_1 + 1)}{\tilde c_0} \Lambda r \text{ for all } r \geq r_2 \,.
	\end{equation}
	Since we assume that $\limsup_{r \to \infty} \frac{V(x) + V(y)}{r} = \infty$, we can indeed choose $r_2$ large enough so that (\ref{eq:proofEulerLyapunov4}) and hence (\ref{eq:6d}) holds. More precisely, we can choose
	\begin{equation}\label{eq:proofEulerLyapunovDefr2}
	r_2 :=\sup \left\{ |x-y| = r : x, y \in \mathbb{R}^d \,, \frac{V(x) + V(y)}{r} < \frac{64C(r_1 + 1)\Lambda}{\lambda \tilde c_0} \right\} \,.
	\end{equation}
	As a consequence, from Theorem \ref{thm:2} we get $\mathbb{E}_{x,y}[\rho_a(X',Y')] \leq (1-c)\rho_a(x,y)$ with $\rho_a(x,y) = (a + \frac{M}{2C}(V(x) + V(y)))\1_{x \neq y} + f_a(|x-y|)$, where
	$c$ is given by (\ref{eq:6f}) and 
	\begin{equation}\label{eq:proofEulerLyapunovDeffa}
	f_a(r) = \int_0^r \varphi(s \wedge r_2) g(s \wedge r_2) ds
	\end{equation}
	with $\varphi$ given by (\ref{eq:proofEulerLyapunovDefVarphi}) and $g$ given by (\ref{eq:6g}).
	Now it only remains to prove the lower bound on the constant $c$. Similarly as in the proof of Theorem \ref{thm:8a}, we have
	\begin{equation*}
	\frac{1}{8} \left( \int_0^{r_2} \frac{1}{\varphi(s)} \sup_{u \in \hat{I}_s} \frac{a + \Phi(u)}{\ua (u)} ds \right)^{-1} \geq \frac{\tilde c_0 h \varphi(r_2)}{8r_2(a + \sqrt{h}) + 4r_2^2} \,,
	\end{equation*}
	since $\varphi(s) \geq \varphi(r_2)$ for $s \leq r_2$. Moreover, due to our choice of $M$ in (\ref{eq:EulerLyapunovDefM}), using $\Phi(r) \leq r$ and (\ref{eq:proofEulerLyapunov4}), we have
	\begin{equation*}
	\frac{\lambda M}{16 C} \frac{V(x) + V(y)}{\Phi(r)} \geq \varphi(r_1) h \Lambda \text{ for all } r \geq r_2 \,.
	\end{equation*}
	This finishes the proof.
\end{proof}

\begin{proof}[Proof of Theorem \ref{thm:8}]
Let $x,y \in \R^d$ and set $r = \lvert x-y \rvert$ and $\hat{r} = \lvert \hat{x} - \hat{y} \rvert$. We assume $h \in (0,h_0]$ where $h_0$ is given by
\eqref{eq:h0general}.\smallskip

We consider at first the case
where $r\ge\mathcal R$. By the choice of $h_0$, we have $L^2h\le K$ for $h\le h_0$. Therefore, for $r\ge \mathcal R$, the concavity of $f$ and  Lemma \ref{lem:6} (v) imply 
\begin{equation}
\label{eq:93} \E_{x,y} [f(R') - f(r)]\, \le\,  \E_{x,y} [R' - r]f'(r)\, \le\, -\frac K2hr f'(r)
\, \le\, -\frac K2e^{-qr_1}hf(r).
\end{equation}
Here, we have used in the last step that $f(r)\le r$ and $f'(r)\ge \exp (-qr_1)$.
The assertion \eqref{eq:45} now follows by the choice of $r_1$ and $q$ in \eqref{eq:35} and \eqref{eq:37}.\medskip

From now on, we assume $r<\mathcal R$. Recall that 
\begin{equation}\label{IrhatEuler}
I_{\hat{r}} = \begin{cases}
(0, \hat{r} + \sqrt{h}) & \text{ if } \hat{r} \leq \sqrt{h}, \\
(\hat{r} - \sqrt{h}, \hat{r}) & \text{ if } \hat{r} > \sqrt{h},
\end{cases}
\end{equation}
cf.\ \eqref{eq:int}, i.e., $u(\hat r)= \sqrt h  \1_{\hat r<\sqrt h}$, $l(\hat r)= \sqrt h  \1_{\hat r\ge \sqrt h}$. Since by Taylor's formula and by concavity of $f$,
\begin{eqnarray*}
f(R')-f(\hat r)&=&\int_{\hat r}^{R'}f'(s)\,ds\ =\ (R'-\hat r)f'(\hat r)\, +\, \int_{\hat r}^{R'}\int_{\hat r}^sf''(t)\, dt\, ds\\
&\le &(R'-\hat r)f'(\hat r)\, +\,\frac 12\left[ ((R'-\hat r)\wedge u(\hat r))\vee (-l(\hat r))\right]^2\, \sup_{I_{\hat r}}f'',
\end{eqnarray*}
we can conclude by Lemma \ref{lem:6} \ref{enum:lem6i} and \ref{enum:lem6ii} that
\begin{eqnarray}
\nonumber \E_{x,y} [f(R') - f(r)] &=&  f(\hat{r} ) - f(r)\, +\, \E_{x,y} [f(R') - f(\hat{r})] \\
&\le &(\hat r -r)f'(r)\, +\, \frac{1}{2} c_0\min (\hat r\sqrt h,h) \sup_{I_{\hat{r}}} f''.\label{eq:9*}
\end{eqnarray}
We are going to show that the expression on the right hand side of \eqref{eq:9*} 
is bounded from above by $-c_2hf(r)$. Note first that by \eqref{eq:407a} and \eqref{eq:407},
\begin{equation}
\label{eq:hatrr} \hat r-r  \le  \min (L,J+L^2h/2)\, hr\ =\ \Lambda hr.
\end{equation}
By the choice of $q$ and $h_0$ in \eqref{eq:37} and \eqref{eq:h0general}, $e^{qh_0L\mathcal R}\le e^{c_0/28}\le e^{1/28}\le 3/2$. Therefore, 
\begin{eqnarray}
\nonumber f'(r)\ =\ e^{-qr}\ = \ e^{q(\hat r-r)}f'(\hat r)& \le & e^{qLhr}f'(\hat r)\ \le\ \frac 32f'(\hat r),\quad\text{and thus}\\
\label{eq:Q1}  (\hat r-r)f'(r) &\le &\frac 32\Lambda h\hat rf'(\hat r).
\end{eqnarray}
Here we have used that $\hat{r} - r \leq \Lambda h \hat{r}$ by \eqref{eq:hatrr} if $r \leq \hat{r}$, whereas for $r > \hat{r}$ \eqref{eq:Q1} is automatically satisfied and hence it holds for all $r < \mathcal{R}$.
Furthermore, by \eqref{eq:h0general}, $hL\le h_0L\le 1/6$. Therefore,
$\hat r -r \le Lhr\le r/6$, and thus
\begin{equation}
\label{eq:Q2}
f(\hat r) \ \ge \ f\left(\frac 56r\right)\ \ge\ \frac 56 f(r),
\end{equation}
because $f$ is increasing and concave with $f(0)=0$. {Note that our choice of the bound $h_0L \leq 1/6$ is to some extent arbitrary and a different choice would lead to $5/6$ above being replaced by a different factor.} By \eqref{eq:9*}, \eqref{eq:Q1} and
\eqref{eq:Q2}, we see that the contractivity condition \eqref{eq:45} holds
provided
\begin{equation}
\label{Q0a}\frac 32\Lambda\, h\hat rf'(\hat r)\, +\, \frac{1}{2} c_0\min (\hat r\sqrt h,h) \sup_{I_{\hat{r}}} f''\ \le\ -\frac 65c_2hf(\hat r).
\end{equation}
Furthermore, by \eqref{eq:hatrr}, $\hat r\le (1+Lh)\mathcal R=r_1$.
Since
$f''(r)=-qe^{-qr}1_{r\le r_1}$ is increasing, \eqref{IrhatEuler} implies
\begin{equation}\label{eq:sup}
\sup_{I_{\hat{r}}} f'' = \begin{cases}
 -qe^{-q(\hat{r} + \sqrt{h})} & \text{ if } \hat{r} \leq \sqrt{h}, \\
-qe^{-q\hat{r}} & \text{ if } \hat{r} > \sqrt{h}.
\end{cases} 
\end{equation}
We now consider these two cases separately:\medskip

\emph{(i) $\hat r>\sqrt h$}. Noting that $f'(\hat r)=e^{-q\hat r}$ and 
$f(\hat r)=(1-e^{-q\hat r})/q\le 1/q$, we see that \eqref{Q0a} is satisfied in this case provided
\begin{equation}
\label{Q0b} 3\Lambda\, \hat r\, -\,  c_0q\ \le\ -\frac {12c_2}{5q}e^{q\hat r}\qquad\text{for }\hat r< r_1.
\end{equation}
We have chosen $q$ in \eqref{eq:37} such that 
$$c_0q\ =\ 7\Lambda\, \mathcal R\ \ge\ 6\Lambda\,r_1.$$
Therefore, the left hand side in \eqref{Q0b} is bounded from above by $-c_0q/2$, and thus \eqref{Q0b} and \eqref{Q0a} are satisfied if
\begin{equation}
\label{Q0c} c_2\ \le \ \frac 5{24}c_0q^2e^{-qr_1}.
\end{equation} 
By \eqref{eq:35} and \eqref{eq:37},
we see that
the constant $c_2$ has been defined in \eqref{eq:c2} in such a way that \eqref{Q0c} holds true, and thus the assertion \eqref{eq:45} is indeed satisfied.
\medskip

\emph{(ii) $\hat r\le\sqrt h$}. Noting that $f'(\hat r)\le 1$ and 
$f(\hat{r}) \leq \hat{r}$, we see by \eqref{eq:sup} that \eqref{Q0a} is satisfied for $\hat r\le\sqrt h$ provided
\begin{equation}
\label{Q0d} 3\Lambda h\,  +\,  \frac {12}{5}c_2h\ \le\ c_0q\sqrt he^{-2q\sqrt h}.
\end{equation}
This condition holds if both
\begin{equation}
\label{Q0e} 2q\sqrt h\le 1/2\qquad\text{and}\qquad
 3e^{1/2}(\Lambda +c_2) \sqrt h\le c_0q.
\end{equation}
It can now be easily verified that our choice of $h_0$ in \eqref{eq:h0general}
ensures that \eqref{Q0e} holds for $h\le h_0$. Indeed, since $q$ is given by \eqref{eq:37}, we have $2q\sqrt h= 14 c_0^{-1}\Lambda\mathcal R\sqrt h\le 1/2$. Moreover, since {$c_2\le 11c_0^{-1}\Lambda^2\mathcal R^2$}, we obtain
$$\frac{c_0q}{3e^{1/2}(\Lambda +c_2)}= \frac{7\Lambda\mathcal R}{3e^{1/2}(\Lambda +c_2)}
\ge \frac{7}{3e^{1/2}}\frac{\mathcal R}{1+11c_0^{-1}\Lambda\mathcal R^2}\ge 
\frac{7}{6e^{1/2}}\min\left(\mathcal R,\frac{c_0}{22\Lambda\mathcal R}\right).$$
Hence \eqref{Q0a} holds true, and thus the assertion \eqref{eq:45} is satisfied in this case as well.
\end{proof}

\section{Proof of results for \MALA}\label{sec:proofof\MALA}
\begin{proof}[Proof of Theorem \ref{thm:10}]
By \eqref{eq:m6}
\begin{equation}
\label{eq:100} X'\ =\ x+ h\, b(x) + \sqrt{h - {h^2}/{4}}\, Z
\end{equation}
where $b(x) = - \frac{1}{2} \nabla U(x)$. By \eqref{eq:d1} and \eqref{eq:m2}, $b$ is Lipschitz continuous on $B_R^{-} \coloneqq \{ x \in \R^d : \| x \|_{-} \le R\}$. 
Therefore conditions \ref{enum:c1} and \ref{enum:c3} in Section \ref{sec:euler} hold when we restrict to $B_R^-$. Moreover, by \eqref{eq:d2}, condition \ref{enum:c2} is satisfied as well for appropriate values of $K$ and $R$ depending on $K_c$ and $\mathcal{R}_c$. 
Noting that $h<2$ implies $ h - \frac{h^2}{4} > \frac{h}{2}$, it is not difficult to see that the proof of Theorem \ref{thm:8} carries over to our slightly modified setup. 
Therefore, similarly to Theorem \ref{thm:8}, we can find for any fixed $R \in (0,\infty)$ a concave strictly increasing function $f$ with $f(0)=0$ and constants $c_2 > 0$, $h_0 > 0$ such that for $h \in (0,h_0)$,
\begin{equation}
\label{eq:101} \E_{x,y} [f(\lvert X' - Y' \rvert)] \le (1-c_2 h) f(\lvert x-y\rvert), \tforall x,y \in B_{R}^{-} \,.
\end{equation}
We now want to apply the perturbation result in Theorem \ref{thm:4}. Setting $d(x,y) = \lvert x-y \rvert$ and $\rho(x,y) = f(\lvert x-y \rvert)$, we see that condition \eqref{eq:27} holds with $c = c_2 h$. Moreover, by Lemma \ref{lem:10b} below, there exists a constant $p > 0$ depending only on $R$, such that for $h_0$ sufficiently small and $h \in (0,h_0)$, condition \eqref{eq:29} is satisfied for any $x,y \in B_R^-$ with $r_0 = \sqrt{h}$. Thus, to apply Theorem \ref{thm:4}, it remains to show that \eqref{eq:28} holds with a constant $b \ge 0$ satisfying 
\begin{equation}
\label{eq:102} b \le {c f(r_0)}/{4} = c_2 {h f(\sqrt{h})}/{4} \,.
\end{equation}
To this end notice that for $x,y \in B_R^{-}$,
\begin{equation}
\label{eq:103} \E_{x,y} [(\lvert \tilde{X} - \tilde{Y} \rvert - \lvert X' - Y' \rvert)^+] \le \E_{x,y} [\lvert \tilde{X} - X' \rvert] + \E_{x,y} [ \lvert \tilde{Y} - Y' \rvert] \,.
\end{equation}
Furthermore, since $\tilde{X} = X'$ if the proposal is accepted and $\tilde{X} = x$ otherwise, we obtain by \eqref{eq:100} and {Lemma \ref{lem:3}}, that for any $x \in B_R^-$ and $h \in (0,2)$, 
\begin{align*}
\E_{x,y} [\lvert \tilde{X} - X' \rvert ] &= \E_{x,y}[ \lvert X' -x \rvert ; \tilde{U} > \alpha_h (x,X') ] \\
&= \E_{x,y} [\lvert X' -x \rvert (1-\alpha_h (x,X'))] \\
&\le h\lvert b(x) \rvert \E_{x,y}[1-\alpha_h (x,X')] + \sqrt{h} \E_{x,y} [\lvert Z \rvert (1-\alpha_h (x,X')]\\
&\le c' h^{\frac{5}{2}} \lvert b(x) \rvert + c'' h^2 \E_{x,y} [\lvert Z \rvert^2]^{\frac{1}{2}},
\end{align*}
where $c'$ and $c''$ are finite constants. Noting that $\E_{x,y} [\lvert Z \rvert^2] = d$ and 
\[
|b(x)| \le d \| b(x) \|_{-} = d {\| \nabla U(x) \|_-}/{2}
\]
we see that there exists a finite constant $c'''$ such that for $x \in B_R^-$ and $h \in (0,2)$
\[
\E_{x,y} [\lvert \tilde{X} - X' \rvert ] \le c''' h^{\frac{3}{2}} (dh + d^{\frac{1}{2}} h^{\frac{1}{2}}) \,.
\]
A corresponding bound holds for $\E_{x,y} [\lvert \tilde{Y} - Y'\rvert]$ with $y \in B_R^-$. Hence, by \eqref{eq:103}, condition \eqref{eq:28} is satisfied with
\[
b = 2c''' h^{\frac{3}{2}} (dh +  d^{\frac{1}{2}} h^{\frac{1}{2}}) \,.
\]
Since the right hand side in \eqref{eq:102} is of order $\Omega(h^{{3}/{2}})$, we conclude that \eqref{eq:102} holds for $dh < h_1$ provided $h_1 \in (0,\infty)$ is chosen sufficiently small. Hence, Theorem \ref{thm:4} applies and by \eqref{eq:30} we obtain
\[
\E_{x,y}[\tilde{f} (\lvert \tilde{X} - \tilde{Y} \rvert)] \le (1-c_3 h) \tilde{f} (\lvert x -y \rvert)
\]
for any $x,y \in B_R^-$ and $h < h_1 d^{-1}$, where $c_3 = \min (\nicefrac{c_2}{8}, \nicefrac{p}{4h})$ and $\tilde{f} (r)  = f(r) + 2bp^{-1} \1_{r>0}$.
\end{proof}

\begin{lemma}\label{lem:10b}
For any fixed $R \in (0,\infty)$ there exist constants $p, h_0 \in (0,\infty)$ such that
\[
\P_{x,y}[\tilde{X} = \tilde{Y} ] \ge p
\]
for any $h \in (0,h_0)$ and $x,y \in B_R^-$ with $\lvert x-y \rvert \le \sqrt{h}$. 
\end{lemma}
\begin{proof}
Let $\hat{x} = x + hb(x)$ where $b(x) = -\frac{1}{2} \nabla U(x)$. By \eqref{eq:d1} and \eqref{eq:m2}, $b$ is Lipschitz continuous on $B_R^{-}$. For $x,y \in B_R^-$ and $h \in (0, L_R^{-1})$
\[
\lvert \hat{x} - \hat{y}\rvert \le (1+hL_R) \lvert x-y \rvert \le 2\lvert x-y\rvert \,,
\]
where $L_R \in (0,\infty)$ is the Lipschitz constant of $b$ on $B_R^-$. Hence, by Lemma \ref{lem:6}, there exists a constant $p_0 \in (0,\infty)$ such that for any $x,y \in B_R^-$ and $h \in (0, L_R^{-1})$
\begin{equation}
\label{eq:10b1} \P_{x,y} [X' = Y'] \ge p_0 \1_{\lvert \hat{x} - \hat{y} \rvert \le 2 \sqrt{h}} \ge p_0 \1_{\lvert x-y \rvert \le \sqrt{h}} \,.
\end{equation}
Furthermore,
\begin{equation}
\label{eq:10b2} \P_{x,y} [\tilde{X} \neq \tilde{Y}] \le \P_{x,y} [X' \neq Y'] + \P_{x,y}[ \tilde{X} \neq X'] + \P_{x,y}[\tilde{Y} \neq Y'] \,.
\end{equation}
By \eqref{eq:10b1}, the first probability on the right hand side is bounded by $1-p_0$ for $\lvert x-y \rvert \le \sqrt{h}$. Moreover, by Lemma \ref{lem:3}, there exists a finite constant $c' \in (0,\infty)$ such that for any $x,y \in B_R^-$ and $h \in (0,2)$
\begin{equation}
\label{eq:10b3} \P_{x,y}[\tilde{X} \neq X'] = \E_{x,y}[1-\alpha_h(x,X')] \le c' h^{\frac{3}{2}} \,.
\end{equation}
A corresponding upper bound holds for $\P_{x,y} [\tilde{Y} \neq Y']$. Hence, by combining \eqref{eq:10b1}, \eqref{eq:10b2} and \eqref{eq:10b3}, we conclude that there exist constants $h_0> 0$ and $p = \frac{p_0}{2}> 0$ such that
\[
\P_{x,y}[\tilde{X} \neq \tilde{Y}] \le 1-p
\]
for any $h \in (0,h_0)$ and $x,y \in B_R^{-}$ with $\lvert x-y \rvert \le \sqrt{h}$.
\end{proof}

\bibliographystyle{amsplain}
\bibliography{bibcoupling}
\end{document}